\documentclass[a4paper,11pt]{amsart}
\usepackage[T1]{fontenc}
\usepackage[utf8]{inputenc}
\usepackage[english]{babel}
\usepackage{etoolbox}
\usepackage{booktabs}
\usepackage{subfigure}
\usepackage{amsmath, amssymb, amsthm, amsfonts}
\usepackage{mathrsfs}
\usepackage{tikz}
\usepackage{graphicx}    
\usepackage{caption}     
\usetikzlibrary{arrows}
\usetikzlibrary{matrix}
\usepackage[left=2.2cm, right=2.2cm, top=2cm]{geometry}
\usepackage{enumerate}
\usepackage{hyperref}
\usepackage{xcolor}

\date{}
\usepackage[font=small,labelfont=bf]{caption}
\theoremstyle{plain}
\usepackage{cleveref}
\newtheorem{theorem}{Theorem}[section] 
\newtheorem{lemma}{Lemma}[section]

\newtheorem{prop}{Proposition}[section]

\nonstopmode\numberwithin{equation}{section}
\theoremstyle{definition}
 
\newtheorem{rema}{Remark}[section]
\newtheorem{assum}{Assumption}[section]

\newtheorem*{ques*}{Question}

\makeatletter
\def\tagform@#1{\maketag@@@{\ignorespaces#1\unskip\@@italiccorr}}
\let\orgtheequation\theequation
\def\theequation{(\orgtheequation)}
\makeatother
\let\orgautoref\autoref

\renewcommand{\autoref}[1]{\def\equationautorefname{}\orgautoref{#1}}

  \usepackage{algorithm}  \usepackage{algpseudocode}  \usepackage{amsmath}
\algrenewcommand\algorithmicrequire{\textbf{Input:}}
            \algrenewcommand\algorithmicensure{\textbf{Output:}}
\usepackage{fancyhdr}
\newcommand\shorttitle{d--SMD method}
\newcommand\authorname{H. Bajpai, G. Mittal, and A.K. Giri}

\usepackage{todonotes}

\begin{document}

\title[d--SMD method]{Damped  stochastic mirror descent for non-smooth ill-posed problems in Banach space}
\author{Harshit Bajpai$^{\dagger}$}
\thanks{$^{\dagger}$Department of Mathematics, Indian Institute of Technology Roorkee, Roorkee, Uttarakhand, 247667, India. 
Email: \texttt{bajpaiharshit87@gmail.com}, \texttt{harshit\_b@ma.iitr.ac.in} (Corresponding author)}

\author{Gaurav Mittal$^{\ddag}$}
\thanks{$^{\ddag}$Defence Research and Development Organization, Near Metcalfe House, Delhi, 110054, India. 
Email: \texttt{gaurav.mittaltwins@yahoo.com}}


\author{Ankik Kumar Giri$^{\intercal}$}
\thanks{$^{\intercal}$Department of Mathematics, Indian Institute of Technology Roorkee, Roorkee, Uttarakhand, 247667, India. 
Email: \texttt{ankik.giri@ma.iitr.ac.in} }

\maketitle
\begin{quote}
{{\em \bf Abstract.}} 
We introduce a damped variant of the stochastic mirror descent method for systems of non-smooth nonlinear ill-posed operator equations in Banach spaces. At each iteration, the method samples a random subset of the equations, performs a mirror-descent step in the dual space with an adaptive step size derived from the discrepancy principle, and applies a damping term that prevents the iterates from drifting away from the initial guess. The scheme accommodates non-smooth forward operators through a family of bounded linear surrogates, admits general $p$-convex penalty functionals which also includes trained convex penalty, allows the noise level to differ from equation to equation, and dispenses with two classical requirements, namely reflexivity of the image space and G\^ateaux differentiability of the forward operator. We study its regularization properties under an \textit{a posteriori} stopping rule by proving first that the iteration terminates after finitely many steps almost surely, and then that the mean-square error of the iterates tends to zero as the noise level vanishes. 
Numerical experiments on a system of non-smooth semilinear inverse source problems and the identification of a diffusion coefficient confirm the theoretical analysis and shows improved convergence results.

\end{quote}

\vspace{.3cm}

\noindent



\section{Introduction}
This study focuses on obtaining a numerical solution to equations of the type
\begin{equation}\label{main}  
    F_{i}(u) = v^{\dagger}_i, \quad i = 1,2,\ldots, M,  
\end{equation}  
where each \( F_i \) is a nonlinear operator defined as \( F_i : \mathfrak{D}(F_i) \subset \mathcal{U} \to \mathcal{V}_i \), mapping from the Banach space \( \mathcal{U} \) to Banach space \( \mathcal{V}_i \). Here, \( \mathfrak{D}(F_i) \) denotes the domain of \( F_i \), and \( v_i^\dagger \) represents the exact data associated with the problem, given a known parameter \( M > 0 \). Such equations
 frequently arise in diverse real-world applications, e.g., they serve as a cornerstone in various tomography methods, which are widely employed in domains such as medical imaging and industrial diagnostics \cite{hanafy1991quantitative,natterer2001mathematics}, making their study both theoretically and practically significant. We assume that at least one solution of \ref{main} exists, although multiple solutions may also exist. However, in practical scenarios, the exact data $v^\dagger:=(v_1^\dagger, v_2^\dagger, \ldots, v_M^\dagger)$ is not accessible due to inherent measurement errors. Instead, we work with an observed noisy dataset $v^{\delta}:=(v_1^\delta, v_2^\delta, \ldots, v_M^\delta)$, which is a measurement of $v^{\dagger}$ that satisfies
 \begin{equation*}\label{noise}
     \|v_i^{\delta} - v_i^{\dagger}\| \leq \varrho_i, \quad i = 1,2, \ldots, M,
 \end{equation*}
where the noise levels $\varrho_i >0$ are known.
 Alternatively, for the product space $\mathcal{V}:= \mathcal{V}_1\times \mathcal{V}_2 \times\cdots \times \mathcal{V}_{M}$ with $\|v\| = \left( \sum_{j=1}^{M} \|v_j\|^s \right)^{1/s}$ for $v \in \mathcal{V}$, \ref{main} can equivalently be rewritten as
\begin{equation}\label{stcombined}
  F(u) = v^{\dagger} , 
\end{equation}
where $F : \mathbb{D}:= \bigcap_{i=1}^M \mathfrak{D}(F_i) \rightarrow \mathcal{V}$ is defined as $$F(u):= (F_1(u), F_2(u), \ldots, F_{M}(u)).$$
 In general, solutions to \ref{main} exhibit high sensitivity to even minor perturbations in the data, making inverse problems such as \ref{stcombined} inherently ill-posed. Consequently, regularization techniques are essential to obtain a stable and meaningful approximation of the solution to \ref{stcombined} (see \cite{engl1996regularization, kaltenbacher2008iterative, kaltenbacher2012Banach} and the references therein). 
  The Landweber iteration stands out as one of the most commonly employed regularization techniques in Banach spaces due to its ease of implementation and computational efficiency. Beyond Landweber-type schemes, faster Newton-type iterative regularization methods equipped with learned or adaptive feature-selection strategies have also been proposed to accelerate convergence for nonlinear inverse problems, see~\cite{Long2024}.

 Let $\Psi:\mathcal{U} \rightarrow (-\infty, \infty]$ be a uniformly convex, lower semi-continuous (\textit{lsc}) function. These functions are helpful in capturing the special characteristics of the exact solution. To reconstruct a solution of \ref{stcombined} that adheres to the characteristics defined by $\Psi$,
a Landweber-type iteration was introduced in \cite{boat2012penality,qinian2013kaczmarz} for  nonlinear  problems, which can be formulated as
\begin{equation}\label{landweberinBanachwithpenality}
\begin{split}
    \zeta_{n+1}^{\delta} &= \zeta_{n}^{\delta} - \omega_{n}^{\delta}F'(u_n^{\delta})^{*}j_{s}^{\mathcal{V}}(F(u_n^{\delta}) - v^{\delta}), \\
    u_{n+1}^{\delta} &= \arg \min_{u \in \mathcal{U}} \{\Psi(u) - \langle \zeta_{n+1}^{\delta}, u \rangle\},
\end{split}
\end{equation}
where $u_0^\delta:= u_0$ and $\zeta_0^\delta := \zeta_0$ are initial guesses, $F'(u)^*$ denotes the adjoint of Fr\'echet derivative $F'(u)$ of $F$ at $u$, $j_{s}^{\mathcal{V}} : \mathcal{V} \rightarrow \mathcal{V}^*$ is the duality mapping with the gauge function $t \rightarrow t^{s-1}$ with $1 < s < \infty$ and $\omega_n^{\delta}$ is the step-size.
The method \ref{landweberinBanachwithpenality} builds upon several previous works  \cite{Hanke,Neuber1995,Neuber2018,Scherzer1996,Schoper2006}. The stopping index $n_\delta$ of the method \ref{landweberinBanachwithpenality} is  defined by the discrepancy principle (\textit{a posteriori})
\begin{equation}\label{posteriori discrepancy}
    \|F(u_{n_{\delta}}^{\delta}) - v^{\delta}\| \leq \tau \delta < \|F(u_{n}^{\delta}) - v^{\delta}\|, \hspace{5mm} 0 \leq n < n_{\delta},
\end{equation}
where $\tau>1$ and $\delta$ denotes the total noise level, as
  $\delta = \big(\sum_{i=1}^M\varrho_{i}^s \big)^{1/s}.$
  


Scherzer \cite{Scherzer1998} introduced a modification to the Landweber method in Hilbert spaces by incorporating a damping term, naming it the iteratively regularized Landweber iteration (IRLI) method. Subsequently in \cite{kaltenbacher2013crates}, Kaltenbacher extended this IRLI method to Banach spaces,  which can be written as 
\begin{equation}\label{IRLM in Banach}
\begin{split}
    \mathcal{J}_s (u_{n+1}^\delta - u_0) &= (1-\alpha_n)\mathcal{J}_s (u_n^\delta - u_0) - \omega_n F'(u_n^\delta)^* j_s(F(u_n^\delta) - v^\delta), \\
    u_{n+1}^\delta &= u_0 + \mathcal{J}_{r}^* (\mathcal{J}_s(u_{n+1}^\delta - u_0)), \quad n\geq 0,
        \end{split}
\end{equation}
where $\alpha_n \in [0,1], s, r >1$ are conjugate indices, $\omega_n >0, \mathcal{J}_s: \mathcal{U} \to \mathcal{U^*}, \mathcal{J}_r^* : \mathcal{U}^* \to \mathcal{U}$ are duality mappings and $j_s(u)$ denotes the single valued selection of the multiple valued mapping $\mathcal{J}_s(u)$.

Compared to the method \ref{landweberinBanachwithpenality}, the convergence rate analysis of \ref{IRLM in Banach} requires fewer assumptions on \(F\). Additionally, the method \ref{IRLM in Banach} converges to a solution that is closer to \(u_0\), whereas achieving similar convergence behavior with the method \ref{landweberinBanachwithpenality} necessitates additional assumptions, see \cite{kaltenbacher2008iterative}. This highlights the advantages of incorporating the additional damping term in \ref{IRLM in Banach}, offering improved convergence properties over \ref{landweberinBanachwithpenality}.
 Recently,  \cite{Mittal2025} extended the IRLI method \ref{IRLM in Banach} by combining a non-smooth, uniformly convex penalty function $\Psi$, which can be written as 
\begin{equation}\label{IRLI Mittal}
    \begin{split}
        \zeta_{n+1}^\delta &= (1-\alpha_n^\delta)\zeta_n^\delta - \omega_n F'(u_n^\delta)^*\mathcal{J}_s(F(u_n^\delta) - v^\delta) + \alpha_n^\delta \zeta_0, \\
        u_{n+1}^\delta &= \nabla \Psi^*(\zeta_{n+1}^\delta), \quad n\geq 0,
    \end{split}
\end{equation}
where $\alpha_n^\delta \geq 0,$ step-sizes $\overline{\omega} \geq \omega_n \geq \underline{\omega}$ for certain $\overline{\omega} , \underline{\omega}>0$ and $\Psi^* : \mathcal{U}^* \to (-\infty, \infty]$ is the Legendre-Fenchel conjugate of $\Psi$ (which is defined in the upcoming Section).
This extension enables the study of convergence using conditional stability estimates (CSE). While CSE offers the advantage of requiring fewer assumptions, variational inequalities (which combine TCC and the source condition) are inherently superior in capturing the geometry of the problem and the interaction between constraints, penalty terms, and the data. In contrast, CSE is limited in this aspect, as it primarily provides a quantitative stability margin under fixed and specific assumptions \cite{Mittal2025}.

It can be observed that, one of the common challenges for the methods \ref{landweberinBanachwithpenality}, \ref{IRLM in Banach} and \ref{IRLI Mittal} and other conventional iterative techniques, is the substantial computational cost incurred at each iteration. This is primarily due to the necessity of processing the entire dataset, which can be vast and computationally demanding.
To mitigate this issue, a promising and efficient strategy is the adoption of stochastic gradient descent (SGD), a method originally introduced by Robbins et al.  \cite{robbins1951stochastic}. Although the SGD method offers notable computational advantages, its theoretical analysis in the context of inverse problems remains relatively unexplored and is still in its early stages. For cases where both \( \mathcal{U} \) and \( \mathcal{V}_i \) are Hilbert spaces, some theoretical advancements have been made. Specifically, the SGD method has been studied for linear inverse problems in \cite{jin2018regularizing, lu2022stochastic}, for nonlinear inverse problems in \cite{jin2020convergence}, and for linear inverse problems in Banach spaces in \cite{jin2023, JinQ2023}. More recently, \cite{GU2024,jin2026} extended the SGD framework to address nonlinear inverse problems in Banach spaces.

In these works, the convergence analysis relies on an \textit{a priori} selection of the stopping index, where the convergence rates are dependent on the smoothness parameter of \( v^\dagger \). However, this parameter is often unknown in practical applications, posing a significant limitation. This challenge underscores the need for developing \textit{a posteriori} stopping rules that can ensure convergence without requiring prior knowledge of the smoothness parameter.

To address this challenge, Jahn and Jin \cite{jahn2020discrepancy} investigated the regularization properties of the SGD method for linear inverse problems in Hilbert spaces under an \textit{a posteriori} stopping rule. More recently, in \cite{GU2024}, the finite-iteration termination property of an \textit{a posteriori} stopping rule for the SGD method with a convex penalty term (SGD-$\Psi$) in Banach spaces was established, although the convergence analysis relied on an \textit{a priori} rule. Furthermore, a heuristic stopping rule was studied for a variant of SGD, specifically iteratively regularized stochastic gradient descent (IRSGD) in \cite{Bajpai2024Hanke}. For the corresponding deterministic Landweber iteration in Banach spaces, an analogous heuristic Hanke--Raus rule has recently been analyzed in \cite{real2024hanke}.

In \cite{GU2024}, the SGD-$\Psi$ method takes the following form
 \begin{equation}\label{SGD Phi}
\begin{split}
    \zeta_{n+1}^{\delta} &= \zeta_{n}^{\delta} - \omega_{n}^{\delta}F_{i_n}(u_n^{\delta})^{*}j_{s}^{\mathcal{V}_{i_n}}(F_{i_n}(u_n^{\delta}) - v_{i_n}^{\delta}), \\
    u_{n+1}^{\delta} &= \arg \min_{u \in \mathcal{U}} \{\Psi(u) - \langle \zeta_{n+1}^{\delta}, u \rangle\},
\end{split}
\end{equation}
where $\omega_n^\delta$ is the suitably chosen step size and $i_n$ is the uniformly drawn index from the index set $\{1, 2, \ldots, M\}.$ It can be noted that \ref{SGD Phi} is a randomized version of \ref{landweberinBanachwithpenality}. For \ref{SGD Phi}, convergence rates have not been shown in general; however, for the linear case, convergence rates have been established by Jin et al. in \cite{JinQ2023}.

It is to be noted that all of the previous convergence analyses of SGD-type methods have been carried out by incorporating the following two conditions in addition to the usual TCC:
\begin{itemize}
    \item [(i)] $\mathcal{V}$ is a uniformly smooth Banach space and $1<s<\infty,$
    \item [(ii)] the forward operator $F$ is Fr\'echet differentiable (that implies that the mapping $u \rightarrow F'(u)$ from $\mathcal{U}$ to $\mathcal{L}(\mathcal{U}, \mathcal{V})$ is continuous),  
\end{itemize}
 where $\mathcal{L}(\mathcal{U},\mathcal{V})$ represents the set of all bounded linear operators from  $\mathcal{U}$ to  $\mathcal{V}$. The conditions mentioned above are generally satisfied by numerous inverse problems. However, it should be noted that there are significant instances in which these conditions may not hold. For example, when data are affected by other types of noise rather than Gaussian noise, such as uniformly distributed noise or impulsive noise, it becomes necessary to select $\mathcal{V}$ which may not be reflexive Banach space and, consequently, may not exhibit uniform smoothness, as discussed in \cite{aster2018piandip, clason2012}. Moreover, certain ill-posed inverse problems involve forward operators $F$ that lacks G\^ateaux differentiability, see \cite{clason2019bouligand}. \\ Apart from the seminal work \cite{clason2019bouligand} for addressing non-smooth ill-posed problems, a wide range of regularization techniques~\cite{fu2023two,JinQ2024-1,Mittal2023,tong2021} have been developed in recent years. In particular, an extension of the Landweber-type method to Banach spaces has been proposed in \cite{rreal2020revisit}, where the classical derivative $F'(u)$ is replaced by a family of bounded linear operators $\{L(u) : u \in \mathfrak{D}(F)\}$ mapping from $\mathcal{U}$ to $\mathcal{V}$. Furthermore, a variant of the stochastic gradient descent (SGD) method tailored for non-smooth settings has been introduced in \cite{Harshitdatadriven}. Related regularization strategies for penalized iterative methods under weakened smoothness of $F$ and/or non-Hilbertian penalty terms $\Psi$ can be found in \cite{zhong2023landweber, Mittal2025}. Collectively, \cite{clason2019bouligand, rreal2020revisit, Harshitdatadriven, fu2023two, Mittal2023, JinQ2024-1, tong2021, zhong2023landweber} constitute the body of work on non-smooth or reduced-smoothness ill-posed problems that our study builds on and extends to the stochastic, mini-batch setting. Motivated by these, our key focus in the present work is to introduce a generalized and damped variant of SGD for solving non-smooth system of equations in Banach spaces. More precisely, we propose a new approach called \textit{damped stochastic mirror descent} (d--SMD), which can be formulated as
 \begin{equation}\label{d--SMD}  
    \begin{split}
        \zeta_{n+1}^{\delta} &= (1-\alpha_n^\delta)\zeta_{n}^{\delta} - \omega_n^{\delta}L_{I_n}(u_n^{\delta})^*j_{s}^{\mathcal{V}_{I_n}}(F_{I_n}(u_n^{\delta})-v_{I_n}^{\delta}) + \alpha_n^\delta\zeta_0, \\
        u_{n+1}^{\delta} &= \arg \min_{u \in \mathcal{U}} \{\Psi(u) - \langle\zeta_{n+1}^{\delta}, u \rangle\},
     \end{split}
\end{equation}
where $\zeta_0^\delta = \zeta_0, u_0^\delta = u_0, \alpha_n^\delta \geq 0$ represents the suitably chosen damping factor, $\omega_n^\delta \geq 0$ denotes the adaptive step-size and $I_n = \{i_1, i_2, \ldots i_\ell\}$ is the randomly selected subset from the index set $\{1, 2, \ldots, M\}$ with batch size $|I_n| =\ell$. 

By using the generalized TCC and specific compactness condition (Assumption \ref{assumption c}), we establish a novel convergence result that eliminates the need for both the reflexivity of $\mathcal{V}$ and the continuity of the mapping $u \rightarrow L(u)$. 
The main highlights of this study are as follows:
\begin{itemize}
    \item Under specific assumptions, we establish the convergence properties of the proposed method \ref{d--SMD} in Banach spaces using an \textit{a posteriori} stopping
    criterion. Notably, we demonstrate its finite iteration termination property.
    The potential non-smoothness and nonlinearity of \(F_i\) and the inclusion of a
    damping term $-\alpha_n^\delta(\zeta_n^\delta - \zeta_0)$ introduce significant
    analytical challenges, making our convergence analysis much more intricate
    compared to prior works such as \cite{GU2024, JinQ2023}.

    \item In contrast to the commonly employed polynomially decreasing or constant
    step-size assumption in previous SGD methods
    \cite{Bajpai2024Hanke, jahn2020discrepancy, jin2020convergence}, we integrate the
    principles of the discrepancy rule into the step-size \(\omega_n^\delta\). This
    approach aims to suppress oscillations in the iterates and mitigate the
    semi-convergence behaviour typically observed in such methods.

\item
For two problems, namely a system of non-smooth semilinear inverse source problems and a diffusion coefficient identification problem, we numerically verify the assumptions underlying the convergence theory by evaluating the $p$-convexity modulus, Bregman radius, operator bounds, and tangential cone constants. Furthermore, we monitor the admissibility conditions for the adaptive step size and damping parameter throughout the iterations, demonstrating that all theoretical requirements are satisfied in practice.

\item 
Numerically, we validate that the adaptive step size eliminates the need for prior knowledge of the operator bound and significantly reduces the computational cost compared with existing iterative regularization methods. The proposed method achieves improved reconstruction accuracy with substantially fewer sparse solves, while the damping strategy further accelerates convergence without sacrificing accuracy.
\end{itemize}

%
%

\textbf{Outline.} Section \ref{sec: 2} introduces the primary definitions and notation necessary for this work. Section \ref{Convergence} is devoted to the detailed convergence analysis of the method  for non-noisy and noisy data.  Section \ref{sec:numerics} is dedicated to vast numerical experiments and discuss the applicability of the method.
Finally, we conclude the manuscript with some potential future research directions in Section~\ref{sec:Conclusion}.
   \section{Preliminaries} \label{sec: 2}
In this section, we present several fundamental principles of Banach spaces and convex analysis, along with their characteristics. Most of these details can be found in \cite{rreal2020revisit, zalinesch202convex}.\\
For a Banach space $\mathcal{U}$, where $\|\cdot\|$ and $\mathcal{U}^*$ represent the norm and dual space, respectively. We express the duality mapping as $\langle \zeta, u \rangle = \zeta(u)$, where $\zeta$ is an element of $\mathcal{U}^*$ and $u \in \mathcal{U}$. For a convex functional $\Psi : \mathcal{U} \rightarrow( -\infty, +\infty],$ we use $\partial\Psi(u)$ to represent the sub-differential of $\Psi$ at $u$, i.e.,
$$\partial\Psi(u) = \{\zeta \in \mathcal{U}^{*} : \Psi(u') - \Psi(u) -  \langle \zeta \; , u' - u\; \rangle   \geq 0 \hspace{2mm} \forall \hspace{2mm} u' \in \mathcal{U} \}.$$
For a given \( \zeta \in \partial \Psi(u) \), the Bregman distance associated with \( \Psi \) at \( u \) in the direction \( \zeta \) is defined as  
\begin{equation*}  
    \mathcal{D}_{\zeta}\Psi(u', u) = \Psi(u') - \Psi(u) - \langle \zeta, u' - u \rangle, \quad \forall u' \in \mathcal{U}.  
\end{equation*}  
It is evident that \( \mathcal{D}_{\zeta}\Psi(u', u) \geq 0 \).  

We define the effective domain of \( \Psi \) as \( \mathfrak{D}(\Psi) = \{u \in \mathcal{U} : \Psi(u) < \infty\} \), and \( \mathfrak{D}(\partial \Psi) = \{u \in \mathfrak{D}(\Psi) : \partial\Psi(u) \neq \emptyset \} \), then the Bregman distance  adheres to the following identity  
\begin{equation}\label{3ptidentity}  
       \mathcal{D}_{\zeta_2}\Psi(u, u_2) - \mathcal{D}_{\zeta_1}\Psi(u, u_1) =  \mathcal{D}_{\zeta_2}\Psi(u_1, u_2) +  \langle \zeta_2 - \zeta_1, u_1 - u \rangle,  
\end{equation}  
for all \( u \in \mathfrak{D}(\Psi) \), \( u_1, u_2 \in \mathfrak{D}(\partial\Psi) \),  and \( \zeta_i \in \partial\Psi(u_i) \) for \( i = 1,2 \).  
We take  the convex functional \( v \mapsto \|v\|^s / s \) over \( \mathcal{V} \) for \( 1 \leq s < \infty \) in our work. Furthermore, \( \mathcal{J}_{s}^{\mathcal{V}} : \mathcal{V} \to 2^{\mathcal{V}^*} \) is used to denote its subdifferential. Specifically, by Asplund’s theorem (\cite{kaltenbacher2012Banach}), the mapping \( \mathcal{J}_{s}^{\mathcal{V}} \) coincides with the duality mapping on \( \mathcal{V} \) associated with the gauge function \( t \mapsto t^{s-1}/s \), for \( 1 < s < \infty \), i.e.,  
\begin{equation*}  
    \mathcal{J}_{s}^{\mathcal{V}}(v) = \{ v^* \in \mathcal{V}^* : \|v^*\|_{\mathcal{V}^*} = \|v\|_{\mathcal{V}}^{s-1}, \quad \langle v^*, v \rangle_{\mathcal{V}^*, \mathcal{V}} = \|v\|_{\mathcal{V}}^s \}.  
\end{equation*}  
while for \( s = 1 \), there holds  
\begin{equation*}  
    \mathcal{J}_{1}^{\mathcal{V}}(v) =  
    \begin{cases}  
        \{ v^* \in \mathcal{V}^* : \|v^*\|_{\mathcal{V}^*} = 1, \quad \langle v^*, v \rangle_{\mathcal{V}^*, \mathcal{V}} = \|v\|_{\mathcal{V}} \} & \text{if } v \neq 0, \\  
        \{ v^* \in \mathcal{V}^* : \|v^*\|_{\mathcal{V}^*} \leq 1 \} & \text{if } v = 0.  
    \end{cases}  
\end{equation*}
 
It is essential to recognize that $\mathcal{J}_{s}^{\mathcal{V}}$ is typically a multi-valued function. The mapping $j_{s}^{\mathcal{V}}$ used in the method \ref{d--SMD} represents a specific choice from $\mathcal{J}_{s}^{\mathcal{V}}$, and it is defined as a single-valued function from $\mathcal{V}$ to $\mathcal{V}^*$. This mapping, denoted as $j_s^{\mathcal{V}}(v)$, is characterized by the property that $j^{\mathcal{V}}_s(v)$ belongs to $\mathcal{J}_{s}^{\mathcal{V}}(v)$ for every $v$ in the Banach space $\mathcal{V}$.
It can be readily observed that
\begin{equation}\label{js ineq}
    \langle j_{s}^{\mathcal{V}}(v), v\rangle_{\mathcal{V}^*, \mathcal{V}} = \|v\|^s_{\mathcal{V}} \hspace{5mm} \text{and} \hspace{5mm} \|j_{s}^{\mathcal{V}}(v)\|_{\mathcal{V}^*} \leq \|v\|_{\mathcal{V}}^{s-1}
\end{equation}
for all $v \in \mathcal{V}$ and $1 \leq s <\infty.$ 
In this context, we adhere to the convention that $0^0=1$.

A Banach space \(\mathcal{U}\) is defined as uniformly smooth if it satisfies the condition  
\(
\lim_{t \rightarrow 0 } \frac{\pi_{\mathcal{U}}(t)}{t} = 0,
\)  
where \(\pi_{\mathcal{U}}(t)\), known as the modulus of smoothness, is given by  
\[
\pi_{\mathcal{U}}(t) = \sup \{\|u' - u\| + \|u' + u\| -2 : \|u'\| =1, \hspace{2mm} \|u\| \leq t\}.
\]  
Notably, when \(\mathcal{U}\) is uniformly smooth, the duality mappings \(\mathcal{J}_s^{\mathcal{U}}\) remain uniformly bounded on bounded sets and are single-valued for every \(1 < s < \infty\).

A proper convex function $\Psi$, as defined above, is said to be uniformly convex if there exists a strictly increasing function $\xi : [0, \infty ) \rightarrow [0, \infty)$ satisfying the inequality
$$ \Psi(\varepsilon u' + (1 - \varepsilon)u) + \varepsilon(1 - \varepsilon)\xi(\|u' - u\|) \leq \varepsilon\Psi(u') +(1-\varepsilon)\Psi(u), \hspace{3mm} \forall \hspace{1mm} u', u \in \mathcal{U},$$
 where  $\xi(0) = 0$, and $\varepsilon \in [0, 1].$
 If we particularly take $\xi(t) = c_0t^p$ for some $p > 1$ and $c_0 > 0$, then $\Psi$ is referred to as $p$-convex.\\
For a function $\Psi : \mathcal{U} \rightarrow (-\infty, \infty]$ that is proper,  convex, and \textit{lsc}, its Legendre-Fenchel conjugate is given by  
$$\Psi^*(\zeta) = \sup_{u \in \mathcal{U}} \{ \langle \zeta \;, u \; \rangle - \Psi(u) \}, \hspace{2mm} \zeta \in \mathcal{U}^*.$$
Assume $\Psi$ is $p$-convex with $p>1$. Then there hold
\begin{equation}\label{inequality in bregmann distance and norm}
   \mathcal{D}_{\zeta}\Psi(\overline{u}, u) \geq c_0\|\overline{u} - u\|^{p}, 
\end{equation}  
\begin{equation}\label{bregandzeta}
\mathcal{D}_{\zeta}\Psi(\overline{u}, u) \leq \frac{1}{p^{*}(2c_0)^{p^{*}-1}} \|\zeta - \overline{\zeta}\|^{p^*}, \overline{\zeta} \in \partial\Psi(\overline{u}),
\end{equation}
where \( 1/p + 1/p^* = 1 \).
Furthermore, $\nabla\Psi^* : \mathcal{U^*} \rightarrow \mathcal{U}$ satisfies 
\begin{equation}\label{gradandzeta}
\|\nabla\Psi^{*}(\zeta_1) - \nabla\Psi^{*}(\zeta_2) \| \leq \left(\frac{\|\zeta_1 -\zeta_2\|}{2c_0} \right)^{\frac{1}{p-1}}, \hspace{2mm} \forall \hspace{2mm} \zeta_1, \zeta_2 \in \mathcal{U^*}.   
\end{equation}
By utilizing the definition of  $u_n^\delta$, we obtain 
 $u_n^\delta = \nabla \Psi^* (\zeta_n^\delta).$ 
We address the annihilator of $\mathcal{N}(F)$ as
   $$\mathcal{N}(F)^{\perp} = \{\zeta \in \mathcal{U}^{*} :\hspace{2mm} \langle \zeta, u \rangle = 0 \hspace{2mm} \forall \hspace{2mm} u \in \mathcal{N}(T)\}$$
   and  the set $\mathcal{B}(u_0, \wp) = \{u \in \mathcal{U}: \|u-u_0\| \leq \wp\}$ for any $\wp >0,$ denotes the closed ball in $\mathcal{U}$. 
   \subsection{Few preliminaries related to probability}
   To analyze the convergence of the method \ref{d--SMD}, it is essential to define a probability space that serves as the foundation for the analysis. Let \(\Pi_M := \{1,2, \dots, M\}\), and let \(\mathcal{S}_M\) denote the \(\sigma\)-algebra comprising all subsets of \(\Pi_M\). At each iteration, an index subset is chosen from \(\Pi_M\) 
  following a uniform probability distribution. Consequently, for any \(n \geq 1\), the random variables \(u_n^\delta\) and \(\zeta_n^\delta\) are naturally defined over the sample space \(\Pi_M^n\), which is endowed with the \(\sigma\)-algebra \(\mathcal{S}_M^{\otimes n}\) and the associated uniform probability measure \(\mathbb{P}_n\), where
\[
\Pi_M^n := \underbrace{\Pi_M \times \cdots \times \Pi_M}_{n \text{ copies}}.
\]  
The Kolmogorov extension theorem \cite{Battacharya2007} guarantees the existence of a unique probability measure \(\mathbb{P}\) on the measurable space \((\Omega, \mathcal{F}) := (\Pi_M^\infty, \mathcal{S}_M^{\otimes \infty})\) such that the family of finite-dimensional distributions \(\mathbb{P}_n\) remains consistent with \(\mathbb{P}\).  The expectation \(\mathbb{E}\) is associated with the probability space \((\Omega, \mathcal{F}, \mathbb{P})\). Given a Banach space \(\mathcal{U}\), the space \(L^s(\Omega, \mathcal{U})\), for \(1 \leq s < \infty\), consists of all \(\mathcal{U}\)-valued random variables \(u\) satisfying the condition \(\mathbb{E}[\|u\|^s] < \infty\). Equipping this space with the norm   
\(
\|u\|_{L^s} := (\mathbb{E}[\|u\|^s])^{1/s}  
\) transforms it into a Banach space. 
For the forthcoming analysis, we consider the natural filtration \(\{\mathcal{F}_n\}\), where \(\mathcal{F}_n := \sigma(I_0, \dots, I_{n-1})\) for \(n \geq 1\). See \cite{JinQ2023} for further details.

We  frequently utilize the following well-established properties of conditional expectation. For any random variables \(\Phi, \Phi_1, \Phi_2,\) and \(\hat{\Phi}\), and real constants \(c_1, c_2\), the following identities hold:  

\begin{itemize}  
    \item [(a)]  
    \(
    \mathbb{E}[\Phi] = \mathbb{E}[\mathbb{E}[\Phi \mid \mathcal{F}_n]].
    \)  
    \item [(b)] \text{Additivity and Linearity}: \\ 
    \(
    \mathbb{E}[\Phi_1 + \Phi_2 \mid \mathcal{F}_n] = \mathbb{E}[\Phi_1 \mid \mathcal{F}_n] + \mathbb{E}[\Phi_2 \mid \mathcal{F}_n],
    \)  
    \\
    \(
    \mathbb{E}[c_1\Phi + c_2 \mid \mathcal{F}_n] = c_1\mathbb{E}[\Phi \mid \mathcal{F}_n] + c_2.
    \)  
    \item [(c)]  
    If \(\Phi \leq \hat{\Phi}\), then  
    \[
    \mathbb{E}[\Phi \mid \mathcal{F}_n] \leq \mathbb{E}[\hat{\Phi} \mid \mathcal{F}_n] \quad \text{almost surely}.
    \]  
\end{itemize}  

Here, \(\mathbb{E}[\Phi \mid \mathcal{F}_n]\) denotes the conditional expectation of \(\Phi\).
\section{Convergence Analysis} \label{Convergence}
In this section, we analyze the convergence of the d--SMD method \ref{d--SMD} when it is terminated using an \textit{a posteriori} stopping rule. A brief outline of this section, which further consists of four subsections, is as follows. The subsection \ref{detailed d--SMD} provides a detailed explanation of the proposed method with an algorithm, while subsection \ref{assumptions} outlines the assumptions underlying our analysis. The convergence of the method is addressed in subsections \ref{subsec: nonnoisy} and \ref{subsec: noisy} for the cases of non-noisy and noisy data, respectively.

 \subsection{d--SMD method}\label{detailed d--SMD}
We recall the proposed method, which is formulated as
\begin{equation}\label{d--SMD 3.1}
     \begin{split}
        \zeta_{n+1}^{\delta} &= (1-\alpha_n^\delta)\zeta_{n}^{\delta} - \omega_n^{\delta}L_{I_n}(u_n^{\delta})^*j_{s}^{\mathcal{V}_{I_n}}(F_{I_n}(u_n^{\delta})-v_{I_n}^{\delta}) + \alpha_n^\delta\zeta_0, \\
        u_{n+1}^{\delta} &= \arg \min_{u \in \mathcal{U}} \{\Psi(u) - \langle\zeta_{n+1}^{\delta}, u \rangle\},
     \end{split}
\end{equation}
where $\zeta_{0}^{\delta} = \zeta_{0}, u_{0}^{\delta} = u_{0}, I_n = \{i_1, i_2, \ldots, i_\ell\}$ is the randomly selected subset from the index set $\{1, 2, \ldots, M\}$ with $|I_n|=\ell, F_{I_n}:= (F_{i_1}, F_{i_2}, \ldots, F_{i_\ell}): \mathcal{U} \to \mathcal{V}_{I_n}$ with $\mathcal{V}_{I_n}:= \mathcal{V}_{i_1}\times \mathcal{V}_{i_2} \times \cdots \times \mathcal{V}_{i_\ell}$ and $j_s^{\mathcal{V}_{I_n}}$ denotes the duality mapping of $\mathcal{V}_{I_n}$. We now fix the notation for the step-size $\omega_n^\delta$ and damping factor $\alpha_n$. In particular, let
 $r_{I_n}^{\delta} = F_{I_n}(u_n^{\delta}) - v_{I_n}^{\delta}$, $ t_n^{\delta} = \|\zeta_{n}^{\delta} - \zeta_0\|$ and, we denote by
\begin{equation}\label{delta In def}
    \delta_{I_n} := \Big(\sum_{i \in I_n} \varrho_i^{s}\Big)^{1/s}
\end{equation}
the (random) local noise level associated with the sampled data $v_{I_n}^{\delta}$. This is the same quantity that enters in the stopping rule \eqref{discrepancy} below. Since $I_n$ ranges over a finite collection of subsets of $\{1,\ldots,M\}$ of size $\ell$, we also introduce the deterministic constant
\begin{equation}\label{delta min def}
    \delta_{\min} := \min_{I \subset \{1,\ldots,M\},\ |I|=\ell} \delta_{I} \leq \delta_{I_n} \quad \text{a.s.},
\end{equation}
which will be used in the proof of Theorem \ref{finite iteration} below. Then the step-size $\omega_n^{\delta}$ for the method \ref{d--SMD 3.1} is defined by
\begin{equation}\label{step}
 \omega_{n}^{\delta} = 
\left\{
    \begin{array}{lr}
     \min \left \{ \frac{\frac{1}{2}(\varsigma_1^{p^* -1}\|r_{I_n}^{\delta}\|^s - \varsigma_{2,n}(t_n^{\delta})^{p^*})^{\frac{1}{p^*-1}}}{\|L_{I_n}(u_n^{\delta})^*j_{s}^{\mathcal{V}_{I_n}}(r_{I_n}^{\delta})\|^p}, \varsigma_3\|r_{I_n}^{\delta}\|^{p-s} \right\} , & \text{if } \|r_{I_n}^\delta\|^s > (\tau \delta_{I_n})^s\\
       0, & \text{if } \|r_{I_n}^\delta\|^s \leq (\tau \delta_{I_n})^s
    \end{array}
\right.
\end{equation}
where $\varsigma_1>0$, and the sequence $\{\varsigma_{2,n}\}$ satisfies the condition
\begin{equation*}
    \varsigma_{2,n}(t_n^{\delta})^{p^*} \leq (\overline{\varsigma}_{2})^{p^*-1}\|r_{I_n}^{\delta}\|^s \leq \varsigma_{1}^{p^* - 1}\|r_{I_n}^{\delta}\|^s,
\end{equation*}
and $\overline{\varsigma}_2 >0, \varsigma_3 > 0.$ Here $p$ is the order of convexity of the functional $\Psi$. We further set
\begin{equation}\label{def varsigma5}
    \varsigma_5 := \frac{1}{2}\big(\varsigma_1^{p^*-1} - \overline{\varsigma}_2^{\,p^*-1}\big)^{\frac{1}{p^*-1}} >0,
\end{equation}
a constant that will appear in the lower bound \eqref{reln bw rs and rp} on $\omega_n^\delta\|r_{I_n}^\delta\|^s$ and, consequently, in the statement of Lemma \ref{lemma 3.2} below.
\begin{rema}
   For a fixed positive real number \(\varsigma_1\), we introduce a sequence of positive real numbers \(\varsigma_{2,n}\) satisfying the condition 
\(
\varsigma_1^{p^*-1}\|r_{I_n}^\delta\|^s - \varsigma_{2,n} (t_n^\delta)^{p^*} > 0.
\) 
The role of the sequence \(\{\varsigma_{2,n}\}\) is crucial, as it ensures that the expression 
\(
\varsigma_1^{p^*-1} \|r_{I_n}^\delta\|^s - \varkappa (t_n^\delta)^{p^*}
\) 
remains positive. Without this sequence, as \(\|r_{I_n}^\delta\|\) diminishes towards zero, the expression would eventually become negative for any fixed positive constant \(\varkappa\), leading to an inadmissible negative step size.
\end{rema}
A constructive choice of $\{\varsigma_{2,n}\}$ ensuring positivity is as follows.  
For any $\rho\in(0,1)$, define
\[
\varsigma_{2,n} =
\begin{cases}
(1-\rho)\,\varsigma_1^{\,p^*-1}\,\dfrac{\|r_{I_n}^\delta\|^s}{(t_n^\delta)^{p^*}}, & t_n^\delta>0, \\[6pt]
0, & t_n^\delta=0.
\end{cases}
\]
Then
\[
\varsigma_1^{p^*-1}\|r_{I_n}^\delta\|^s - \varsigma_{2,n}(t_n^\delta)^{p^*}
= \rho\,\varsigma_1^{p^*-1}\|r_{I_n}^\delta\|^s >0,
\]
so the step-size in \eqref{step} is always admissible. Moreover, choosing
\(
\overline{\varsigma}_2=(1-\rho)^{\tfrac{1}{p^*-1}}\varsigma_1
\)
ensures that
\[
\varsigma_{2,n}(t_n^\delta)^{p^*} \leq (\overline{\varsigma}_2)^{p^*-1}\|r_{I_n}^\delta\|^s,
\]
which matches the required bound.

Further, we assume that the damping factor $\alpha_n$ introduced in \ref{d--SMD 3.1} satisfies
        \begin{equation}\label{condonalpha}
            \alpha_n^\delta \leq \min \{\varsigma_4 \omega_n^{\delta}\|F_{I_n}(u_{n}^{\delta}) - v_{I_n}^{\delta}\|^s(t_{n}^{\delta})^{-1}, 2^{\frac{1-p^*}{p^*}}(\omega_{n}^{\delta}\varsigma_{2,n})^{\frac{1}{p^*}}, \varsigma_6 \omega_n^\delta\},
        \end{equation}
whenever $\omega_{n}^{\delta}, t_{n}^{\delta} \neq 0$, with $\varsigma_4, \varsigma_6 >0.$

Each iteration in \ref{d--SMD 3.1} operates through a two-step procedure. The first step performs an iteratively regularized stochastic gradient descent in the dual space \( \mathcal{U}^* \), leveraging its structure to optimize the solution. The second step utilizes the mirror map associated with \( \Psi \) to project the updated element from the dual space back into the original primal space \( \mathcal{U} \), ensuring that the solution remains consistent with the desired feature representation.

To obtain an approximate solution using \ref{d--SMD 3.1}, it is crucial to implement an appropriate termination criterion. We design an \textit{a posteriori} stopping rule which determines the stopping index $n_\delta$ by
\begin{equation}\label{discrepancy}
n_\delta:= \min \{n \in \mathbb{N} : \|F_{I_n}(u_{n}^{\delta}) - v_{I_n}^{\delta}\|^s \leq (\tau \delta_{I_n})^s\},
\end{equation}
where $I_n \subset \{1, 2, \ldots, M\}$ with $|I_n|=\ell$ and $\delta_{I_n}$ is the local noise level defined in \eqref{delta In def}. It is important to note that, $n_\delta$ depends on the random selection of the subset \( I_n \) and random iterate \( u_n^\delta \), hence \( n_\delta \) itself becomes a random variable. This makes the theoretical analysis more challenging. The stopping rule \ref{discrepancy} has been extensively analyzed, with particular focus on the approximation properties of \( u_{n_\delta}^\delta \) as \( \delta \to 0 \).
\begin{rema}
   The proposed discrepancy principle~\eqref{discrepancy} is evaluated only on the sampled block $I_n$. Consequently, the algorithm may terminate prematurely if a randomly selected block satisfies the stopping criterion, even when the remaining blocks do not. To mitigate this issue, we incorporate an active-set strategy in our numerical experiments. Further details on the implementation of the active-set strategy are provided in Section~\ref{sec:numerics}.
\end{rema}
Based on these facts and results the algorithm for d--SMD method is given as follows.
\begin{algorithm}[htbp]
\caption{Damped stochastic mirror descent with penalty $\Psi$ (d--SMD)}
\label{alg:dsmd}
\begin{algorithmic}[1]
\Require Operators $F_1, \ldots, F_M$; $u_0\in\mathcal{U}$, $\zeta_0\in\partial\Psi(u_0)$; noise levels
  $\varrho_1,\dots,\varrho_M$; block size $\ell$; $\tau>1$;
  constants $\varsigma_1,\varsigma_3,\varsigma_4,\varsigma_6>0$ and $\rho\in(0,1)$
\Ensure $u_{n_\delta}^{\delta}$
\State $n\gets 0$, \quad $\zeta_0^{\delta}\gets\zeta_0$, \quad $u_0^{\delta}\gets u_0$
\Loop
  \State draw $I_n\subset\{1,\dots,M\}$ at random with $|I_n|=\ell$
  \State $r_{I_n}^{\delta}\gets F_{I_n}(u_n^{\delta})-v_{I_n}^{\delta}$, \quad
         $\delta_{I_n}\gets\big(\sum_{i\in I_n}\varrho_i^{\,s}\big)^{1/s}$
         \Comment{local noise level \eqref{delta In def}}
  \If{$\|r_{I_n}^{\delta}\|^{s}\le(\tau\delta_{I_n})^{s}$}
     \State $n_\delta\gets n$; \Return $u_{n_\delta}^{\delta}$
     \Comment{stopping rule \eqref{discrepancy}; here $\omega_n^{\delta}=0$}
  \EndIf
  \State $t_n^{\delta}\gets\|\zeta_n^{\delta}-\zeta_0\|$, \quad
         $g_n\gets L_{I_n}(u_n^{\delta})^{*}\,j_s^{\mathcal{V}_{I_n}}\!\big(r_{I_n}^{\delta}\big)$
  \State $\varsigma_{2,n}\gets
         (1-\rho)\,\varsigma_1^{\,p^*-1}\|r_{I_n}^{\delta}\|^{s}/(t_n^{\delta})^{p^*}$
         if $t_n^{\delta}>0$, else $\varsigma_{2,n}\gets 0$
  \State $\displaystyle
         \omega_n^{\delta}\gets\min\left\{
         \frac{\tfrac12\big(\varsigma_1^{\,p^*-1}\|r_{I_n}^{\delta}\|^{s}
               -\varsigma_{2,n}(t_n^{\delta})^{p^*}\big)^{\frac1{p^*-1}}}{\|g_n\|^{p}},\
         \varsigma_3\|r_{I_n}^{\delta}\|^{p-s}\right\}$
         \Comment{step size \eqref{step}}
  \State $\displaystyle
         \alpha_n^\delta \gets\min\Big\{\varsigma_4\,\omega_n^{\delta}
         \|r_{I_n}^{\delta}\|^{s}(t_n^{\delta})^{-1},\
         2^{\frac{1-p^*}{p^*}}\big(\omega_n^{\delta}\varsigma_{2,n}\big)^{\frac1{p^*}},\
         \varsigma_6\,\omega_n^{\delta}\Big\}$ \quad
         ($\alpha_n^\delta \gets0$ if $t_n^{\delta}=0$)
  \State $\zeta_{n+1}^{\delta}\gets(1-\alpha_n^\delta)\,\zeta_n^{\delta}
         -\omega_n^{\delta}\,g_n+\alpha_n^\delta\,\zeta_0$
         \Comment{dual step}
  \State $u_{n+1}^{\delta}\gets\arg\min_{u\in\mathcal{U}}
         \big\{\Psi(u)-\langle\zeta_{n+1}^{\delta},u\rangle\big\}$
         \Comment{mirror map}
  \State $n\gets n+1$
\EndLoop
\end{algorithmic}
\end{algorithm}

\subsection{Assumptions}\label{assumptions}
The following assumptions are necessary for the analysis of d--SMD method.
\begin{assum}\label{assumption a}
    Let $\Psi : \mathcal{U} \rightarrow (-\infty, \infty]$ be a proper, \textit{lsc} function that is $p$-convex for some $p > 1$. There exists a constant $c_0 > 0$ such that, for all $\overline{u}, u \in \mathcal{U}$ and any $t \in [0,1]$, the following inequality holds  
\begin{equation*}
    \Psi(t\overline{u} + (1-t)u) \leq t\Psi(\overline{u}) + (1-t)\Psi(u) - c_0 t(1-t)\|\overline{u} - u\|^{p}.
\end{equation*}
\end{assum}
\begin{assum}\label{assumption b}
\hspace{1mm}
    \begin{itemize}
    \item [(i)] $F$ is weakly closed over $\mathbb{D}$, where $\mathbb{D}= \bigcap_{i=1}^M \mathfrak{D}(F_i) $.
    \item [(ii)] There exist $\wp >0, u_0 \in \mathcal{U}$ and $\zeta_0 \in \partial\Psi(u_0)$ such that $\mathcal{B}(u_0, 3\wp) \subset \mathbb{D}$ and \ref{main} has a solution $\hat{u}$ satisfying 
    \begin{equation}\label{sigmacond}
       \mathcal{D}_{\zeta_0}\Psi(\hat{u}, u_0) \leq c_0 \wp^p,
    \end{equation}where $c_0$ is the same constant as in Assumption~\ref{assumption a}.
    \item[(iii)] There exists a family of bounded linear operators $\{L(u): \mathcal{U} \to \mathcal{V}\}_{u \in \mathcal{B}(u_0, 3\wp)}$ satisfying 
\[
\|L(u)\|_{\mathcal{L}(\mathcal{U}, \mathcal{V})} \leq B, \quad \forall u \in \mathcal{B}(u_0, 3\wp),
\]
for some constant $B > 0$. Additionally, a constant $\eta \geq 0$ with $\eta < 1$ exists such that 
\begin{equation}\label{tangential}
    \|F_{I_n}(\overline{u}) - F_{I_n}(u) - L_{I_n}(u)(\overline{u}-u)\| \leq \eta \|F_{I_n}(\overline{u}) - F_{I_n}(u)\|
\end{equation}
for any $\overline{u}, u \in \mathcal{B}(u_0, 3\wp)$ and $I_n \subset \{1, 2, \dots, M\}$.  
 This inequality is known as \textit{generalized tangential cone condition} (GTCC).
\end{itemize}

\end{assum}
To remove the requirement of continuity of the mapping $u \rightarrow L(u)$  and uniform smoothness of $\mathcal{V}$, we incorporate the following compactness condition, as introduced in \cite{clason2019bouligand}.
 \begin{assum}\label{assumption c}
      There exists a Banach space $\mathcal{W}$ satisfying the condition
 \begin{equation}\label{compactnesscondition}
 \bigcup_{u \in \mathcal{B}(u_0, 3\wp)} range(L(u)^*) \subset \mathcal{W},
 \end{equation}
where $\mathcal{W}$ is compactly embedded into $\mathcal{U}^*.$ Additionally, there exists a constant $\hat{B}>0$ such that for all $u \in \mathcal{B}(u_0, 3\wp)$, $\|L(u)^*\|_{\mathcal{L}(\mathcal{V}^*, \mathcal{W})} \leq \hat{B}.$ 
 \end{assum}
We seek a solution $u^{\dagger}$ of \ref{main} that possesses the property
 \begin{equation}\label{udagger}
 \mathcal{D}_{\zeta_{0}}\Psi(u^{\dagger}, u_0) = \min_{u \in \mathfrak{D}(F)}\{\mathcal{D}_{\zeta_{0}}\Psi(u, u_0) : F_{i}(u) = v_i^{\dagger}, i=1, 2, \ldots M\}.  
 \end{equation}
The Assumptions \ref{assumption a}, \ref{assumption b}  guarantee the  uniqueness of \(u^{\dagger}\) (see \cite[Lemma 3.2]{qinian2013kaczmarz}).

Next, we present an important result that will serve as a fundamental component in the forthcoming analysis.

\begin{lemma}\label{lemma 3.2}
    Let $\varsigma_1, \varsigma_4 >0$ be chosen such that 
    \begin{equation}\label{codn2}
        C_1 := 1 - \eta -\frac{1+\eta}{s}\left(s-1 - \frac{1}{\tau^s}\right) -\frac{1}{p^*}\left(\frac{\varsigma_1}{2c_0}\right)^{p^* - 1} -6\wp \varsigma_4 >0.
    \end{equation}
    Furthermore, let Assumptions $\ref{assumption a}$ and $\ref{assumption b}$ hold. Then for any solution $\overline{u} \in \mathcal{B}(u_0, 3\wp)$ of $\ref{main}$, we have 
    \begin{equation*}
        \mathcal{D}_{\zeta^\delta_{n+1}}\Psi(\overline{u}, u^\delta_{n+1}) \leq \mathcal{D}_{\zeta^\delta_n}\Psi(\overline{u}, u^\delta_n),
    \end{equation*}
    \begin{equation*}
         \mathbb{E}\left[\mathcal{D}_{\zeta_{n+1}^\delta}\Psi(\overline{u}, u^\delta_{n+1})\right] - \mathbb{E}\left[\mathcal{D}_{\zeta_{n}^\delta}\Psi(\overline{u}, u^\delta_{n}) \right] \leq -\frac{C_2 \ell }{M}\mathbb{E}\left[ \|F(u^\delta_n) - v^\delta\|^p \right], 
    \end{equation*}
   where $C_2:= C_1 \min \{\varsigma_5 \hat{B}^{-p} , \varsigma_3\}$ and $n\geq 0$.
\end{lemma}
\begin{proof}
   Let $\overline{u}$ be any solution of \ref{main} in $\mathcal{B}(u_0, 3\wp).$ If $u_n^\delta \in \mathcal{B}(u_0, 3\wp)$, then from  \ref{3ptidentity} and \ref{bregandzeta}, it follows that
    \begin{equation}\label{dec1}
        \begin{split}
            \mathcal{D}_{\zeta^\delta_{n+1}}\Psi(\overline{u}, u^\delta_{n+1}) - \mathcal{D}_{\zeta^\delta_{n}}\Psi(\overline{u}, u^\delta_{n}) &= \mathcal{D}_{\zeta^\delta_{n+1}}\Psi(u^\delta_n, u^\delta_{n+1}) + \langle \zeta^\delta_{n+1} - \zeta^\delta_n, u^\delta_n - \overline{u} \rangle \\
            & \leq \frac{1}{p^*(2c_0)^{p^*-1}}\|\zeta^\delta_{n+1}-\zeta^\delta_n\|^{p^*} + \langle\zeta^\delta_{n+1}  - \zeta^\delta_n, u^\delta_n - \overline{u} \rangle.
        \end{split}
    \end{equation}
   We proceed by estimating each part of \ref{dec1} separately, starting with the first part as
    \begin{equation}\label{3.3}
        \begin{split}
             \|\zeta^\delta_{n+1} - \zeta^\delta_n\|^{p^*} &= \|\alpha_n^\delta \zeta^\delta_n + \omega^\delta_nL_{I_n}(u_n^\delta)^*j_{s}^{\mathcal{V}_{I_n}}(r^\delta_{I_n}) - \alpha_n^\delta\zeta_0\|^{p^*} \\
             & \leq 2^{p^*-1}\left[(\omega^\delta_n)^{p^*}\|L_{I_n}(u^\delta_n)^*j_{s}^{\mathcal{V}_{I_n}}(r^\delta_{I_n})\|^{p^*} + (\alpha_{n}^\delta)^{p^*}\|\zeta^\delta_n - \zeta_0\|^{p^*}\right].
        \end{split}
    \end{equation}
  Furthermore, based on the definition of $\omega^\delta_n$ in \ref{step}, we have
   \begin{equation*}
      (\omega_n^\delta)^{p^*-1} \leq  \frac{\frac{1}{2^{p^*-1}}(\varsigma_1^{p^* -1}\|r^\delta_{I_n}\|^s - \varsigma_{2,n}(t^\delta_n)^{p^*})}{\|L_{I_n}(u^\delta_n)^*j_{s}^{\mathcal{V}_{I_n}}(r^\delta_{I_n})\|^{p(p^*-1)}},
   \end{equation*}
   which means 
   \begin{equation}\label{3.4}
      2^{p^* -1} (\omega^\delta_n)^{p^*-1}\|L_{I_n}(u^\delta_n)^*j_{s}^{\mathcal{V}_{I_n}}(r^\delta_{I_n})\|^{p(p^*-1)} \leq \varsigma_1^{p^* -1}\|r^\delta_{I_n}\|^s - \varsigma_{2,n}(t^\delta_n)^{p^*}.
   \end{equation}
   By substituting \ref{3.4} in \ref{3.3}, we get
   \begin{equation}\label{norm}
       \|\zeta^\delta_{n+1} - \zeta^\delta_n\|^{p^*} \leq \omega^\delta_n \varsigma_1^{p^*-1}\|r^\delta_{I_n}\|^s + \left(2^{p^*-1}(\alpha_n^\delta)^{p^*} -\varsigma_{2, n}\omega^\delta_n \right)(t^\delta_n)^{p^*}.
   \end{equation}
   Further, we deduce an estimate for the second term as
       \begin{align*}
           \nonumber \langle \zeta^\delta_{n+1} - \zeta^\delta_n, u^\delta_n - \overline{u} \rangle &= -\langle \alpha_n^\delta \zeta^\delta_n + \omega^\delta_nL_{I_n}(u^\delta_n)^*j_{s}^{\mathcal{V}_{I_n}}(r_{I_n}^\delta) - \alpha_n^\delta \zeta_0, u^\delta_n - \overline{u} \rangle \\ \nonumber
           &= - \langle \alpha_n^\delta (\zeta^\delta_n -\zeta_0), u^\delta_n -\overline{u} \rangle +\omega^\delta_n \langle L_{I_n}(u^\delta_n)^*j_{s}^{\mathcal{V}_{I_n}}(r_{I_n}^\delta),  \overline{u} - u^\delta_n\rangle.
     \end{align*}    
From GTCC \ref{tangential} and \ref{js ineq}, we have
     \begin{align}\label{ineq in zeta n}
      \nonumber  \langle \zeta^\delta_{n+1} - \zeta^\delta_n, u^\delta_n - \overline{u} \rangle   & \leq \alpha_n^\delta \|\zeta^\delta_n - \zeta_0\|\|u^\delta_n - \overline{u}\| - (1-\eta)\omega^\delta_n\|r_{I_n}\|^s + (1+\eta)\omega_n^\delta \delta_{I_n}\|r_{I_n}^\delta\|^{s-1} \\ 
         & \leq 6\wp \alpha_n^\delta \|\zeta^\delta_n -\zeta_0\| - (1-\eta)\omega^\delta_n\|r_{I_n}^\delta\|^s + (1+\eta)\omega_n^\delta \delta_{I_n}\|r_{I_n}^\delta\|^{s-1}.
       \end{align}
     By using \ref{condonalpha} and Young's inequality $xy \leq \frac{x^p}{p} + \frac{y^q}{q}$ for any $x, y \geq 0$ and conjugate indices $p, q>1$, respectively, in the first and third term of \ref{ineq in zeta n},  we get    
  \begin{equation}\label{inner}
     \langle \zeta^\delta_{n+1} - \zeta^\delta_n, u^\delta_n - \overline{u} \rangle  \leq 6\wp\varsigma_4\omega^\delta_n \|r_{I_n}^\delta\|^{s} - (1-\eta)\omega^\delta_n\|r_{I_n}^\delta\|^s+  (1+\eta)\omega_n^\delta\left(\frac{(s-1)\|r_{I_n}^\delta\|^s}{s} + \frac{\delta^s_{I_n}}{s} \right).
       \end{equation}
       Incorporating \ref{norm} and \ref{inner} in \ref{dec1}, we deduce that
 \begin{align*}
            \mathcal{D}_{\zeta^\delta_{n+1}}\Psi(\overline{u}, u^\delta_{n+1}) &- \mathcal{D}_{\zeta^\delta_{n}}\Psi(\overline{u}, u^\delta_{n}) \\
             &\leq \frac{1}{p^*(2c_0)^{p^*-1}}\left[\omega^\delta_n \varsigma_{1}^{p^* -1}\|r_{I_n}^\delta\|^s + (2^{p^*-1}(\alpha_n^\delta)^{p^*} - \omega^\delta_n\varsigma_{2,n})\|\zeta^\delta_n - \zeta_0\|^{p^*} \right]\\
             & + \left( 6\wp\varsigma_4\omega^\delta_n + \frac{(1+\eta)(s-1)}{s}\omega_n^\delta  - (1-\eta)\omega^\delta_n\right)\|r_{I_n}^\delta\|^s + \frac{(1+\eta)\omega_n^\delta \delta^s_{I_n}}{s}.
    \end{align*}
  This along with \ref{condonalpha} implies that
$$ 
            \mathcal{D}_{\zeta^\delta_{n+1}}\Psi(\overline{u}, u^\delta_{n+1}) - \mathcal{D}_{\zeta^\delta_{n}}\Psi(\overline{u}, u^\delta_{n})\hspace{120mm} $$
             $$\leq \omega_n^\delta\left(\frac{ \varsigma_{1}^{p^*-1}}{p^*(2c_0)^{p^*-1}} + 6\wp\varsigma_4 + \frac{(1+\eta)(s-1)}{s}  - (1-\eta) \right)\|F_{I_n}(u^\delta_n) - v_{I_n}^{\delta}\|^s 
             +\frac{(1+\eta)\omega_n^\delta \delta^s_{I_n}}{s}.$$
   By applying \ref{codn2} and  the definition of $\omega_n^\delta$ in the last estimate, we obtain
    \begin{equation}\label{decminus}
        \mathcal{D}_{\zeta^\delta_{n+1}}\Psi(\overline{u}, u^\delta_{n+1}) - \mathcal{D}_{\zeta^\delta_{n}}\Psi(\overline{u}, u^\delta_{n})
             \leq -C_1\omega^\delta_n\|F_{I_n}(u^\delta_n) - v_{I_n}^{\delta}\|^s .
    \end{equation}
    This implies that \( \mathcal{D}_{\zeta^\delta_{n+1}}\Psi(\overline{u}, u^\delta_{n+1}) \leq \mathcal{D}_{\zeta^\delta_{n}}\Psi(\overline{u}, u^\delta_{n})\) for $n\geq 0$.\\
    To prove the next assertion, we use the fact that $\{L(u)^*\}$ is bounded (so is $\{L_{I_n}(u)^*\} \hspace{1mm} \forall I_n \subset \{1,2,\ldots, M\})$ and the  definition of $\omega_n^\delta$, as
    $$         \omega^\delta_{n} \geq \min \left \{ \frac{(\varsigma_1^{p^* -1}\|r^\delta_{I_n}\|^s - \varsigma_{2,n}(t^\delta_n)^{p^*})^{\frac{1}{p^*-1}}}{2\hat{B}^{p}\|r^\delta_{I_n}\|^{p(s-1)}}, \varsigma_3\|r_{I_n}^{\delta}\|^{p-s} \right\}\hspace{70mm} $$ $$\geq \min \left \{ \frac{(\varsigma_1^{p^* -1} -\overline{\varsigma}_{2}^{p^*-1})^{\frac{1}{p^*-1}}\|r^\delta_{I_n}\|^{\frac{s}{p^*-1}} }{2\hat{B}^{p}\|r^\delta_{I_n}\|^{p(s-1)}}, \varsigma_3\|r_{I_n}^{\delta}\|^{p-s} \right\}
           =\min \left \{ \frac{(\varsigma_1^{p^* -1} -\overline{\varsigma}_{2}^{p^*-1})^{\frac{1}{p^*-1}} }{2\hat{B}^{p}}, \varsigma_3 \right\}\|r^\delta_{I_n}\|^{p-s}. $$ This means that $$
           \omega_n \|r^\delta_{I_n}\|^s  \geq\min \left \{ \frac{(\varsigma_1^{p^* -1} -\overline{\varsigma}_{2}^{p^*-1})^{\frac{1}{p^*-1}} }{2\hat{B}^{p}}, \varsigma_3 \right\}\|r^\delta_{I_n}\|^{p}.$$
    Recalling the definition of $\varsigma_5$ from \eqref{def varsigma5}, i.e., $(\varsigma_1^{p^* -1} -\overline{\varsigma}_{2}^{p^*-1})^{\frac{1}{p^*-1}}  = 2\varsigma_5,$ we further arrive at
    \begin{equation}\label{reln bw rs and rp}
        \omega^\delta_n \|r^\delta_{I_n}\|^s  \geq \min \left \{\varsigma_5 \hat{B}^{-p} , \varsigma_3 \right\}\|r^\delta_{I_n}\|^{p}.
    \end{equation}
    By engaging \ref{reln bw rs and rp} in \ref{decminus}, we derive that 
    \begin{equation}\label{decminus 1} 
        \mathcal{D}_{\zeta^\delta_{n+1}}\Psi(\overline{u}, u^\delta_{n+1}) - \mathcal{D}_{\zeta^\delta_{n}}\Psi(\overline{u}, u^\delta_{n}) 
             \leq -C_2 \|F_{I_n}(u^\delta_n) - v_{I_n}^{\delta}\|^p,
    \end{equation}
 where the constant $C_2>0$   is same as defined in the statement.
    By taking the expectation value of \ref{decminus 1} with respect to filtration \(\mathcal{F}_n = \sigma(I_0, \ldots, I_{n-1})\) and leveraging  the measurability of $u^\delta_n$ with respect to $\mathcal{F}_n$ and denoting the summation over all subsets $ I \subset \{1, \ldots, M\}$ by $\sum_{I: |I| = \ell}$, we have
    \begin{align*}
        \mathbb{E}\left[\mathcal{D}_{\zeta^\delta_{n+1}}\Psi(\overline{u}, u^\delta_{n+1}) - \mathcal{D}_{\zeta^\delta_{n}}\Psi(\overline{u}, u^\delta_{n}) | \mathcal{F}_n \right]
             &\leq - C_2 \mathbb{E}\left[ \|F_{I_n}(u^\delta_n) - v_{I_n}^{\delta}\|^p |\mathcal{F}_n\right].      
    \end{align*}
    Finally, by taking the full expectation, we get
    \begin{align*}
       & \mathbb{E}\left[\mathcal{D}_{\zeta^\delta_{n+1}} \Psi(\overline{u}, u^\delta_{n+1})\right] - \mathbb{E}\left[\mathcal{D}_{\zeta^\delta_{n}}\Psi(\overline{u}, u^\delta_{n}) \right]  \leq - C_2 \mathbb{E}[ \|F_{I_n}(u^\delta_n) - v_{I_n}^{\delta}\|^p ] \\
       & \quad \quad \quad \quad \quad = - C_2 \mathbb{E}\left[\mathbb{E}[ \|F_{I_n}(u^\delta_n) - v_{I_n}^{\delta}\|^p \mid \mathcal{F}_n ]\right] 
         =  -C_2\mathbb{E}\left[\frac{1}{\binom{M}{\ell}} \sum_{I:|I|=\ell} \|F_{I}(u^\delta_n) - v_{I}^{\delta}\|^p\right]\\
        & \quad \quad \quad \quad \quad=  -\frac{C_2}{\binom{M}{\ell}} \sum_{I:|I|=\ell} \sum_{i \in I}\mathbb{E}\left[ \|F_{i}(u^\delta_n) - v_{i}^{\delta}\|^p \right] 
        = -\frac{C_2\binom{M-1}{\ell-1} }{\binom{M}{\ell}}\sum_{i=1}^M \mathbb{E} \left[\|F_{i}(u^\delta_n) - v_{i}^{\delta}\|^p \right] 
          \end{align*}\begin{equation}\label{new1}
                     = -\frac{C_2 \ell }{M}\mathbb{E}\left[ \|F(u^\delta_n) - v^{\delta}\|^p \right].  \hspace{58mm}
\end{equation}
  This concludes the proof.
    \end{proof}
    The next result gives an almost sure finite-iteration termination property of the rule \ref{discrepancy}.
    \begin{theorem}\label{finite iteration}
        Let $n_\delta$ be defined by \textit{a posteriori} rule $\ref{discrepancy}$ and  let all the assumptions of Lemma $\ref{lemma 3.2}$ hold. Then for $\tau >1$, the d--SMD method $\ref{d--SMD 3.1}$ stops, almost surely, in finite steps. That is  \[\mathbb{P}(\varepsilon) =0,\]
        where the event $\varepsilon$ is given by \(\varepsilon:=\{\|F_{I_n}(u_n^\delta) - v^{\delta}_{I_n}\|^s > (\tau \delta_{I_n})^s \  \forall \ n \in \mathbb{N}\}.\)
    \end{theorem}
    \begin{proof}
       By taking the conditional expectation of \ref{decminus} and computing the sum from $k$ to $m-1$, we have
        \[C_1 \sum_{n=k}^{m-1} \mathbb{E}\left[\omega_n^\delta\|F_{I_n}(u^\delta_n) - v_{I_n}^{\delta}\|^s | \mathcal{F}_n \right] \leq  \left(\mathbb{E}\left[\mathcal{D}_{\zeta^\delta_k}\Psi(\overline{u}, u^\delta_k) | \mathcal{F}_k\right] - \mathbb{E}\left[\mathcal{D}_{\zeta^\delta_ m}\Psi(\overline{u}, u^\delta_m) | \mathcal{F}_m\right]\right).\]
        Thus for any integer $t \geq 0,$ there holds
        \begin{equation}\label{finite n 1}
            C_1 \sum_{n=0}^{t} \mathbb{E}[\omega_n^\delta\|F_{I_n}(u^\delta_n) - v_{I_n}^{\delta}\|^s | \mathcal{F}_n] \leq  \mathcal{D}_{\zeta_0}\Psi(\overline{u}, u_0) .
        \end{equation}
 Let $\chi_\varepsilon$ be the characteristic function of the 
 event $\varepsilon.$ Note that $\delta_{I_n}$, as defined in \eqref{delta In def}, depends on the randomly sampled subset $I_n$ and is therefore itself a random variable. In particular it {cannot} be treated as an $\mathcal{F}_n$-measurable constant that may be factored out of a conditional expectation. To circumvent this, we use the deterministic lower bound $\delta_{\min}$ introduced in \eqref{delta min def}, which satisfies $\delta_{I_n} \geq \delta_{\min}$ almost surely, regardless of which subset $I_n$ is sampled. Then, by definition of $\omega_n^\delta$ and \ref{reln bw rs and rp}, we have
        \begin{equation}\label{finite n 2}
            \mathbb{E}[\omega_n^\delta\|F_{I_n}(u^\delta_n) - v_{I_n}^{\delta}\|^s | \mathcal{F}_n] \geq \mathbb{E}[\omega_n^\delta\|F_{I_n}(u^\delta_n) - v_{I_n}^{\delta}\|^s \chi_\varepsilon | \mathcal{F}_n]\qquad\qquad\qquad\qquad \end{equation}
           $$ \geq \min \left \{\varsigma_5 \hat{B}^{-p} , \varsigma_3 \right\} \mathbb{E}[\|r^\delta_{I_n}\|^{p} \chi_\varepsilon | \mathcal{F}_n] \ \geq \frac{C_2}{C_1}\,\mathbb{E}\big[(\tau \delta_{I_n})^p \chi_\varepsilon \mid \mathcal{F}_n\big] \geq \frac{C_2}{C_1} (\tau \delta_{\min})^p\, \mathbb{E}[\chi_\varepsilon \mid \mathcal{F}_n], $$
           where the last two inequalities use, respectively, the pointwise bound $\|r_{I_n}^\delta\|^p\chi_\varepsilon \geq (\tau\delta_{I_n})^p\chi_\varepsilon$ (valid on $\varepsilon$ by the very definition of $\varepsilon$ and \eqref{discrepancy}) and the deterministic, $\mathcal{F}_n$-measurable bound $\delta_{I_n}\geq \delta_{\min}$. Taking the total expectation on both sides (tower property) gives
           \[\mathbb{E}\big[\omega_n^\delta\|F_{I_n}(u_n^\delta)-v_{I_n}^{\delta}\|^s\big] \geq \frac{C_2}{C_1}(\tau\delta_{\min})^p\,\mathbb{P}(\varepsilon).\]
Combining \ref{finite n 1} and \ref{finite n 2} yields
\[C_2(t+1)(\tau \delta_{\min})^p \mathbb{P}(\varepsilon) \leq  \mathcal{D}_{\zeta_0}\Psi(\overline{u}, u_0), \forall\ t \geq 0. \]
It follows that \( P(\varepsilon) \to 0 \) as \( t \to \infty \), since $\delta_{\min}>0$ is a fixed constant not depending on $t$. Consequently, we conclude that \( P(\varepsilon) = 0 \). This implies that the complement event \( \varepsilon^c \) occurs with probability one. This completes the proof.

    \end{proof}
   \subsection{Convergence for non-noisy case}\label{subsec: nonnoisy}
  In this section, we analyze the convergence behavior of non-noisy iterates (i.e., for $\delta =0$) generated by the method \ref{d--SMD 3.1}. For this purpose, we first recall a crucial general convergence result from \cite[Proposition 3.6]{qinian2013kaczmarz}, which can be extended to SMD-type iterative methods. This extended result is formalized in the following proposition, providing a general convergence principle for any sequences \(\{u_n\} \subset \mathcal{U}\) and \(\{\zeta_n\} \subset \mathcal{U}^*\) under specific conditions.

\begin{prop}\label{mainprop}
Suppose assumptions $\ref{assumption a}$ and $\ref{assumption b}$ hold. Let $\{u_n\} \subset \mathcal{B}(u_0, 3\wp)$ and $\{\zeta_n\} \subset \mathcal{U}^*$ be such that
\begin{enumerate}
    \item [(i)] $\zeta_n \in \partial \Psi(u_n)$ for all $n$.
    \item [(ii)] For any solution $\overline{u}$ of $\ref{main}$ in $\mathcal{B}(u_0, 3\wp) \cap \mathfrak{D}(\Psi)$, the sequence $\{\mathbb{E}[\mathcal{D}_{\zeta_n}\Psi(\overline{u}, u_n)]\}$ is monotonically decreasing.
    \item [(iii)] $\lim_{n \rightarrow \infty}\|F(u_n) - v^{\dagger}\| = 0$ almost surely.
    \item [(iv)] There is a subsequence $\{n_k\}$ with $n_k \rightarrow \infty$ such that, for any solution $\overline{u}$ of \ref{main} in $\mathcal{B}(u_0, 3\wp)\cap D(\Psi)$ and for all $l < k$  there holds
    \begin{equation*}
        \lim_{l \rightarrow \infty} \sup_{k \geq l} |\mathbb{E}[\langle \zeta_{n_k} - \zeta_{n_l}, u_{n_k} - \overline{u}\rangle]| = 0.
    \end{equation*}
\end{enumerate}
Then there exists a solution $\overline{u}$ of $\ref{main}$ in $\mathcal{B}(u_0, 3\wp) \cap \mathfrak{D}(\Psi)$ such that
\begin{equation*}
    \lim_{n \rightarrow \infty} \mathbb{E}[\mathcal{D}_{\zeta_n}\Psi(\overline{u}, u_n)] = 0.
\end{equation*}
If in addition $\zeta_{n+1} - \zeta_n \in \mathcal{N}(L(u^{\dagger}))^{\perp}$ for all $n,$ then $\overline{u} = u^{\dagger}.$
\end{prop}
The demonstration of the above proposition is based on the concepts and techniques outlined in \cite[Theorem 3.5]{GU2024} and \cite[Theorem 3.2]{JinQ2023}. Therefore, the proof is omitted here for brevity.

Our main objective is to demonstrate that for every sequence $\{(\zeta_n, u_n)\} \in \Im_{\varsigma_1, \varsigma_3} (\zeta_0, u_0)$ (cf. subsection \ref{subsec: noisy} for the meaning of this notation), the sequence $\{u_n\}$ converges to a solution of equation \ref{main} as $n \to \infty$. To establish this result, we will rely on Proposition \ref{mainprop}. In the following result, we prove assertions (ii) and (iii) of Proposition \ref{mainprop}.

  \begin{prop}\label{prop 2}
        Let $\varsigma_1, \varsigma_4 >0$ be chosen such that 
    \begin{equation}\label{codn 4.1 1}
       C_3 := 1 - \eta  
       -\frac{1}{p^*}\left(\frac{\varsigma_1}{2c_0}\right)^{p^* - 1} -6\wp \varsigma_4 >0.
    \end{equation}
   Furthermore, suppose that assumptions $\ref{assumption a}$ and $\ref{assumption b}$ hold. Then for any solution $\overline{u} \in \mathcal{B}(u_0, 3\wp)$ of $\ref{main}$, we have 
    \begin{equation*}
        \mathcal{D}_{\zeta_{n+1}}\Psi(\overline{u}, u_{n+1}) \leq \mathcal{D}_{\zeta_n}\Psi(\overline{u}, u_n),
    \end{equation*}
    \begin{equation*}
         \mathbb{E}\left[\mathcal{D}_{\zeta_{n+1}}\Psi(\overline{u}, u_{n+1})\right] - \mathbb{E}\left[\mathcal{D}_{\zeta_{n}}\Psi(\overline{u}, u_{n}) \right] \leq -\frac{C_4 \ell }{M}\mathbb{E}\left[ \|F(u_n) - v^\dagger\|^p \right], 
    \end{equation*}
   where $C_4:= C_3 \min \{\varsigma_5 \hat{B}^{-p} , \varsigma_3\}.$
    Moreover, we have
    \begin{equation*}
        \lim_{n \rightarrow \infty}\|F(u_n) - v^{\dagger}\|^p = 0 \quad \text{almost surely}.
    \end{equation*}
  \end{prop}
  \begin{proof}
      For $\delta=0,$ \ref{dec1} becomes
      \begin{equation}\label{dec non noisy}
    \begin{split}
        \mathcal{D}_{\zeta_{n+1}}\Psi(\overline{u}, u_{n+1}) - \mathcal{D}_{\zeta_{n}}\Psi(\overline{u}, u_{n}) \leq \frac{1}{p^*(2c_0)^{p^*-1}}\|\zeta_{n+1}-\zeta_{n}\|^{p^*} + \langle\zeta_{n+1}  - \zeta_{n}, u_n - \overline{u} \rangle.
    \end{split}
\end{equation}
Further, \ref{norm} and \ref{ineq in zeta n} take the form
\begin{equation}\label{norm non noisy}
    \|\zeta_{n+1} - \zeta_n\|^{p^*} \leq \omega_n \varsigma_1^{p^*-1}\|r_{I_n}\|^s + (2^{p^*-1}\alpha_n^{p^*} -\varsigma_{2, n}\omega_n)(t_n)^{p^*}
\end{equation}
and
\begin{equation}\label{zeta non noisy}
    \langle\zeta_{n+1}  - \zeta_{n}, u_n - \overline{u} \rangle  \leq 6\wp \alpha_n \|\zeta_n -\zeta_0\| - (1-\eta)\omega_n\|F_{I_n}(u_n)-v^\dagger_{I_n}\|^s.
\end{equation}
Inserting \ref{norm non noisy} and \ref{zeta non noisy} in \ref{dec non noisy} along with \ref{condonalpha} to obtain
\begin{align*} 
    \mathcal{D}_{\zeta_{n+1}}\Psi(\overline{u}, u_{n+1}) - \mathcal{D}_{\zeta_{n}}\Psi(\overline{u}, u_{n}) 
\leq \omega_n\left(\frac{\varsigma_{1}^{p^*-1}}{p^*(2c_0)^{p^*-1}} + 6\wp\varsigma_4  - (1-\eta) \right)\|F_{I_n}(u_n) - v^\dagger_{I_n}\|^s .
\end{align*}
By applying \ref{codn 4.1 1}, we obtain 
\begin{align*} 
    \mathcal{D}_{\zeta_{n+1}}\Psi(\overline{u}, u_{n+1}) - \mathcal{D}_{\zeta_{n}}\Psi(\overline{u}, u_{n}) 
\leq -C_3\omega_n\|F_{I_n}(u_n) - v^\dagger_{I_n}\|^s.
\end{align*}
Thus, we have \(\mathcal{D}_{\zeta_{n+1}}\Psi(\overline{u}, u_{n+1}) \leq \mathcal{D}_{\zeta_n}\Psi(\overline{u}, u_n)\). Next, on the similar lines of \ref{new1}, we can derive the corresponding result for the noise-free case as 
\begin{equation}\label{non-noisy reln in d and f}
    \mathbb{E}[\mathcal{D}_{\zeta_{n+1}}\Psi(\overline{u}, u_{n+1})] - \mathbb{E}[\mathcal{D}_{\zeta_n}\Psi(\overline{u}, u_n)] \leq -\frac{C_4 \ell}{M} \mathbb{E}\left[\|F(u_n) - v^\dagger\|^p\right],
\end{equation}
where the constant \(C_4>0\) is same as defined in the statement. Consequently, utilizing  \ref{non-noisy reln in d and f}, we deduce that for any integer $0\leq n < m$ 
    \begin{equation}\label{eqn:for_event}
        \sum_{k=n}^{m-1} \mathbb{E}[\|F(u_n) - v^{\dagger}\|^p] \leq  \frac{M}{C_4\ell}\left(\mathbb{E}[\mathcal{D}_{\zeta_n}\Psi(\overline{u}, u_n)] - \mathbb{E}[\mathcal{D}_{\zeta_ m}\Psi(\overline{u}, u_m)]\right).
    \end{equation}
     In particular,  $\mathbb E[\sum_{n=0}^\infty \|F(u_n)-v^\dagger\|^p]<\infty$, implying that $\sum_{n=0}^\infty \|F(u_n)-v^\dagger\|^p<\infty$ almost surely, and therefore $\|F(u_n)-v^\dagger\|^p\to0$ almost surely.
  \end{proof}
From Section \ref{sec: 2}, we note that \( u_n = \nabla \Psi^{*}(\zeta_n) \), which implies \( \zeta_n \in \partial \Psi(u_n) \). Consequently, part \((i)\) of Proposition \ref{mainprop} is fulfilled. In the next result, we finalize the main conclusion of this subsection by addressing part \((iv)\) of Proposition \ref{mainprop}.

 \begin{theorem}\label{thm alpha n sumcond}
        Let all the assumptions of Proposition $\ref{prop 2}$ hold. Assume that the sequence $\{\alpha_n\}$ is such that 
    \begin{equation}\label{sumcond on alpha n}
        \sum_{n=0}^{\infty} \mathbb{E}\left[\alpha_n \|\zeta_0 - \zeta_n\|\right] < \infty.
    \end{equation}
     Then there exists a $\overline{u}$ which satisfies $\ref{main}$ in $\mathcal{B}(u_0, 3\wp)\cap \mathfrak{D}(\Psi)$ such that
    \begin{equation*}
         \lim_{n \rightarrow \infty} \mathbb{E}[\mathcal{D}_{\zeta_n} \Psi(\overline{u}, u_n)] = 0 \quad \text{and} \quad \lim_{n \rightarrow \infty} \mathbb{E}[\|u_n - \overline{u}\|] =0.
     \end{equation*}   
 \end{theorem}
   \begin{proof} Let us choose a strictly increasing sequence of integers $\{n_k\}$  by setting $n_0=0$ and defining $n_k$ as the first integer satisfying
       \begin{equation*}
           n_k \geq n_{k-1} +1\hspace{5mm} \text{and} \hspace{5mm} \mathbb{E}[\|F(u_{n_k}) - v^{\dagger}\|] \leq \mathbb{E}[\|F(u_{n_{k-1}}) - v^{\dagger}\|].
       \end{equation*}
     Using this sequence, we have that 
       \begin{equation*}
           \mathbb{E}[\|F(u_{n_k}) - v^{\dagger}\|] \leq \mathbb{E}[\|F(u_{n}) - v^{\dagger}\|], \hspace{5mm} 0 \leq n \leq n_k.
       \end{equation*}
       For any integer $0 \leq l < k < \infty,$ 
       $$
           |\langle\zeta_{n_k} - \zeta_{n_l}, u_{n_k} - \overline{u}\rangle| \leq \sum_{n = n_l}^{n_k -1}|\langle\zeta_{n+1} - \zeta_{n}, u_{n_k} - \overline{u} \rangle|\hspace{90mm} $$
           $$ \leq \sum_{n = n_l}^{n_k -1}|\omega_n \langle  j_s^{\mathcal{V}_{I_n}} (F_{I_n}(u_n) - v_{I_n}^\dagger), L_{I_n}(u_n)(\overline{u} - u_{n_k}) \rangle|  + \sum_{n = n_l}^{n_k -1}|\langle \alpha_n(\zeta_{n} - \zeta_{0}), u_{n_k} - \overline{u} \rangle|\hspace{20mm} $$    
       $$ \leq (1+\eta) \sum_{n = n_l}^{n_k -1}\omega_n \|F_{I_n}(u_n) - v_{I_n}^{\dagger}\|^{s-1} (2\|F_{I_n}(u_n) - v_{I_n}^{\dagger}\| + \|F_{I_n}(u_{n_k}) - v_{I_n}^{\dagger}\|)   +\sum_{n = n_l}^{n_k -1}\alpha_n\| \zeta_{n} - \zeta_{0}\|  \|u_{n_k} - \overline{u} \|.$$
           We utilize the fact that if \( u_{n_k} \in \mathcal{B}(u_0, 3\wp) \), then \( u_{n_k} \in \mathcal{B}(\overline{u}, 6\wp) \). Consequently, by taking the total expectation, we can obtain
            $$
            |\mathbb{E}[\langle\zeta_{n_k} - \zeta_{n_l}, u_{n_k} - \overline{u}\rangle]|  \leq 3(1+\eta) \sum_{n = n_l}^{n_k -1}\mathbb{E}[\omega_n \|F_{I_n}(u_n) - v_{I_n}^{\dagger}\|^{s}] +6\wp\sum_{n = n_l}^{n_k -1}\mathbb{E}[\alpha_n\| \zeta_{n} - \zeta_{0}\|]  $$
             $$ \leq 3(1+\eta)\varsigma_3 \sum_{n = n_l}^{n_k -1} \mathbb{E}[\|F_{I_n}(u_n) - v_{I_n}^{\dagger}\|^{p}] +6\wp\sum_{n = n_l}^{n_k -1}\mathbb{E}[\alpha_n\| \zeta_{n} - \zeta_{0}\|] \hspace{35mm}$$              $$ = \frac{3(1+\eta)\varsigma_3 \ell}{M} \sum_{n = n_l}^{n_k -1} \mathbb{E}[\|F(u_n) - v^{\dagger}\|^{p}] +6\wp\sum_{n = n_l}^{n_k -1}\mathbb{E}[\alpha_n\| \zeta_{n} - \zeta_{0}\|] \hspace{36mm}$$
             $$ \leq \frac{ 3(1+\eta)\varsigma_3}{C_4}\left(\mathbb{E}[\mathcal{D}_{\zeta_{n_l}}\Psi(\overline{u}, u_{n_l})] - \mathbb{E}[\mathcal{D}_{\zeta_{n_k}}\Psi(\overline{u}, u_{n_k})] \right) +6\wp\sum_{n = n_l}^{n_k -1}\mathbb{E}[\alpha_n\| \zeta_{n} - \zeta_{0}\|].\ \  \ $$ Further, from Proposition \ref{prop 2}, it is clear that $\{\mathbb{E}[\mathcal{D}_{\zeta_n}(\Psi(\overline{u}, u_n)]\}$ is a monotonically decreasing sequence. Hence there exist some $\kappa >0$ such that $\lim_{n \rightarrow \infty}\mathbb{E}[\mathcal{D}_{\zeta_n}(\Psi(\overline{u}, u_n)] = \kappa.$ Hence for a fixed $l$, we have
            \begin{equation*}
                \sup_{k \geq l} |\mathbb{E}[\langle\zeta_{n_k} - \zeta_{n_l}, u_{n_k} - \overline{u}\rangle]| \leq 6\wp\sum_{n = n_l}^{n_k -1}\mathbb{E}[\alpha_n\| (\zeta_{n} - \zeta_{0})\|] + \frac {3(1+\eta)\varsigma_3}{C_4}(\mathbb{E}[\mathcal{D}_{\zeta_{n_l}}\Psi(\overline{u}, u_{n_l})] - \kappa ).
            \end{equation*}
            By taking limit $l \rightarrow \infty$ and using \ref{sumcond on alpha n}, we get
            \begin{equation}
                \lim_{l \rightarrow \infty}\sup_{k \geq l} |\mathbb{E}[\langle\zeta_{n_k} - \zeta_{n_l}, u_{n_k} - \overline{u}\rangle]| \leq \lim_{l \rightarrow \infty} \frac {3(1+\eta)\varsigma_3}{C_4}(\mathbb{E}[\mathcal{D}_{\zeta_{n_l}}\Psi(\overline{u}, u_{n_l})] - \kappa ) = 0.
            \end{equation}
            Therefore,  using Proposition \ref{mainprop} and $p$-convexity of $\Psi$, we get the desired result (as parts (i)-(iii) of Proposition \ref{mainprop} are already satisfied).
   \end{proof}
   \begin{rema}[On the choice of $\alpha_n$]\label{rema:alpha choice}
   Condition \eqref{sumcond on alpha n} is not an additional restriction to be imposed separately. It is automatically satisfied by {any} sequence $\{\alpha_n\}$ obeying the standing admissibility bound \eqref{condonalpha} (with $\delta = 0$, i.e. $\alpha_n \leq \varsigma_4 \omega_n\|F_{I_n}(u_n) - v_{I_n}^\dagger\|^s (t_n)^{-1}$ whenever $t_n \neq 0$), which is required throughout the paper for the step-size to be admissible in the first place. Indeed, multiplying this bound by $t_n = \|\zeta_n - \zeta_0\|$ gives the pointwise estimate
   \[
        \alpha_n \|\zeta_n - \zeta_0\| \leq \varsigma_4\, \omega_n \|F_{I_n}(u_n) - v_{I_n}^{\dagger}\|^s \qquad (\text{trivially true when } t_n = 0).
   \]
   Taking expectations and summing over $n$, and using the same telescoping argument as in the proof of Theorem \ref{finite iteration} (see \eqref{finite n 1}), applied here to the non-noisy decrease estimate of Proposition \ref{prop 2}, yields
   \[
        \sum_{n=0}^{\infty} \mathbb{E}[\alpha_n \|\zeta_0 - \zeta_n\|] \leq \varsigma_4 \sum_{n=0}^{\infty} \mathbb{E}[\omega_n\|F_{I_n}(u_n) - v_{I_n}^{\dagger}\|^s] \leq \frac{\varsigma_4}{C_3}\, \mathcal{D}_{\zeta_0}\Psi(\overline{u}, u_0) < \infty,
   \]
   where $C_3>0$ is the constant from \eqref{codn 4.1 1}. Hence \eqref{sumcond on alpha n} holds for every admissible choice of $\alpha_n$, e.g. $\alpha_n = \varsigma_6 \omega_n$ or $\alpha_n = \min\{\varsigma_6\omega_n,\, c(n+1)^{-1-\epsilon}\}$ for arbitrary $c>0, \epsilon>0$, and no further restriction on $\{\alpha_n\}$ beyond \eqref{condonalpha} is required for Theorem \ref{thm alpha n sumcond} above to apply.
   \end{rema}
   \subsection{Regularization Framework}\label{subsec: noisy}
   In this subsection, our primary objective is to demonstrate the regularization properties of the iterates generated by the d--SMD method, when stopped using  \textit{a posteriori} stopping rule \ref{discrepancy}. To demonstrate the convergence of \(u_{n_\delta}^{\delta}\) to a solution of \(\ref{main}\), it is crucial to analyze the behavior of \((\zeta_n^{\delta}, u_n^{\delta})\) for each \(n\) as \(\delta \to 0\).\\ However, due to lack of continuity in the mapping \(u \mapsto L(u)\) and the discontinuity of the mapping \(u \in \mathcal{U} \mapsto j_s^{\mathcal{V}}(u) \in \mathcal{V}^*\), we cannot generally expect \((\zeta_n^{\delta}, u_n^{\delta})\) to converge to a single point as \(\delta \to 0\). 
Instead, as \(\delta \to 0\), the pair \((\zeta_n^{\delta}, u_n^{\delta})\) may have multiple limit points. By selecting one such limit point for each \(n\), we can construct an iterative sequence \(\{(\zeta_n, u_n)\} \subset \mathcal{U}^* \times \mathcal{U}\), which corresponds to the noise-free case.\\ Importantly, all such iterative sequences share certain properties, which will be outlined in the following discussion.
 We define $\Im_{\varsigma_1,\varsigma_3} (\zeta_0, u_0)$ as the collection of all potential sequences $\{(\zeta_n, u_n)\}$ in $\mathcal{U}^* \times \mathcal{U}$ that can be generated from $(\zeta_0, u_0)$ by
\begin{equation}\label{construction}
    \zeta_{n+1} = (1-\alpha_n)\zeta_n + \alpha_n\zeta_0 - \omega_n\psi_{I_n} \hspace{2mm} \text{and} \hspace{2mm} u_{n+1} = \arg \min_{u \in \mathcal{U}}\{\Psi(u) - \langle\zeta_{n+1},u \rangle\},
\end{equation}
where
\begin{equation}\label{omega}
    \omega_{n} = \begin{cases}
         \min \left \{ \frac{\frac{1}{2}(\varsigma_1^{p^* -1}\|r_{I_n}\|^s - \varsigma_{2,n}(t_n)^{p^*})^{\frac{1}{p^*-1}}}{\|L_{I_n}(u_n)^*j_{s}^{\mathcal{V}_{I_n}}(r_{I_n})\|^p}, \varsigma_3\|r_{I_n}\|^{p-s} \right\} &  \text{if } F_{I_n}(u_n) \neq v_{I_n}^{\dagger}, \\
         0 &  \text{if } F_{I_n}(u_n) = v_{I_n}^{\dagger},
    \end{cases}  
\end{equation}
\(r_{I_n}:= F_{I_n}(u_n) - v_{I_n}, t_n:= \|\zeta_n - \zeta_0\|\)
and $\psi_{I_n} \in \mathcal{U}^*$ satisfies the properties
\begin{align}\label{psinsatisfy}
    \langle \psi_{I_n}, \overline{u} - u_n \rangle &\leq -(1 - \eta) \|F_{I_n}(u_n)-v_{I_n}^{\dagger}\|^{s},
    \end{align}
    \begin{equation}\label{psi 2}
          |\langle \psi_{I_n}, \overline{u} - u \rangle| \leq (1 + \eta) \|F_{I_n}(u_n)-v_{I_n}^{\dagger}\|^{s-1}(2\|F_{I_n}(u_n) - v_{I_n}^{\dagger}\| + \|F_{I_n}(u) - v_{I_n}^{\dagger}\|)
    \end{equation}

for any integer $n,$ and $u \in \mathcal{B}(u_0, 3\wp)$ and any solution $\overline{u}$ of \ref{main} in $\mathcal{B}(u_0, 3\wp).$ \\
It is worth noting that \ref{psinsatisfy} and \ref{psi 2} hold when 
\[
\psi_{I_n} = L_{I_n}(u_n)^* j_s^{\mathcal{V}_{I_n}} (F_{I_n}(u_n) - v_{I_n}^\dagger).
\]
Nevertheless, for this specific selection of \( \psi_{I_n} \), one can leverage assumptions \ref{assumption a} and \ref{assumption b}, along with the similar argument used in \cite{qinian2013kaczmarz}, to establish that $ u_n \in \mathcal{B}(u_0, 3\wp)$ for all  $n \geq 0.$ Thus the set  $\Im_{\varsigma_1,\varsigma_3} (\zeta_0, u_0)$ is non-empty. Before delving into this main result, we first focus on establishing the path-wise stability property of the d--SMD method. In particular, the following lemma shows that for any \(n \geq 0\) and any given sample path \((I_1, \ldots, I_{n-1}) \in \mathcal{F}_n\), the iterates satisfy the convergence conditions \(\|u_n^\delta - u_n\| \to 0\) and \(\|\zeta_n^\delta - \zeta_n\| \to 0\) as the noise level \(\delta \to 0\).

  \begin{lemma}\label{Lemma 3.1}
 Let $\tau >1,$ and $ \varsigma_1, \varsigma_4 >0$ be chosen such that $C_1$ defined by $\ref{codn2}$ be positive and assumptions $\ref{assumption a}, \ref{assumption b}$ and $\ref{assumption c}$ hold.
 Let $\{v^{\delta_j}\}$ be a sequence satisfying $\|v^{\delta_j} - v\| \leq \delta_j$ with $\delta_j \rightarrow 0 $ as $j \rightarrow \infty,$ and let $n_{\delta_j}$ be the stopping index as defined in $\ref{discrepancy}$. Then by selecting an appropriate subsequence of $\{v^{\delta_j}\}$ if necessary, there is a sequence $\{(\zeta_n, u_n)\} \in \Im_{\varsigma_1, \varsigma_3}(\zeta_0, u_0)$ such that
    \begin{equation*}
        u_{n}^{\delta_j} \rightarrow u_n \hspace{5mm} \text{and} \hspace{5mm}\zeta_{n}^{\delta_j} \rightarrow \zeta_n \hspace{2mm} \text{as} \hspace{1mm} j \rightarrow \infty
    \end{equation*}
    for all $0 \leq n \leq \liminf_{j \rightarrow \infty} n_{\delta_j}$ and any path $(I_10, \ldots, I_{n-1}) \in \mathcal{F}_n.$\\
   Additionally, if $\mathcal{N}(L(u^{\dagger})) \subset \mathcal{N}(L(u))$ for all $u \in \mathcal{B}(u_0, 3\wp),$ there also holds
    $\zeta_{n+1} - \zeta_0 \in \mathcal{N}(L(u^{\dagger}))^{\perp}$ for all n.
\end{lemma}
\begin{proof}
   Let $\hat{n} = \liminf_{j \rightarrow \infty} n_{\delta_j}.$ For $n=0,$ we note that $\zeta_{0}^{\delta_j} = \zeta_0$ and $u_{0}^{\delta_j} = u_0,$ thus the result holds trivially. Using the diagonal sequence argument, it is sufficient to demonstrate that if \( (\zeta_n, u_n) \in \mathcal{U}^* \times \mathcal{U} \) is generated for some \( n < \hat{n} \), with \( \zeta_n^{\delta_j} \to \zeta_n \) and \( u_n^{\delta_j} \to u_n \), then, by selecting a suitable subsequence of \( \{v^{\delta_j}\} \) if necessary, we can generate \( (\zeta_{n+1}, u_{n+1}) \in \mathcal{U}^* \times \mathcal{U} \) such that \( \zeta_{n+1}^{\delta_j} \to \zeta_{n+1} \) and \( u_{n+1}^{\delta_j} \to u_{n+1} \) as \( j \to \infty \).
 For the case $\hat{n},$ one may use \ref{construction} with \(\psi_{I_n} = L_{I_n}(u_n)^* j_s^{\mathcal{V}_{I_n}} (F_{I_n}(u_n) - v_{I_n}^\dagger)\) to produce a whole sequence $\{(\zeta_n, u_n)\} \in \Im_{\varsigma_1,\varsigma_3} (\zeta_0, u_0)$ once $\{(\zeta_n, u_n) : 0 \leq n \leq \hat{n}\}$ are generated.
 To generate $(\zeta_{n+1}, u_{n+1}),$ we consider the following two possible cases: \\
    \textit{Case-$1$: $r_{I_n} =0$.} In this case $\omega_n = 0$ and from \ref{step}, we have
   \(
        \omega_n^{\delta_j} \leq \varsigma_3\|r_{I_n}^{\delta_j}\|^{p-s}.
\)
    This implies  
    $$
        \omega_n^{\delta_j}\|L_{I_n}(u_n^{\delta_j})^*j_s^{\mathcal{V}_{I_n}}(F_{I_n}(u_n^{\delta_j}) - v_{I_n}^{\delta_j})\|  \leq \varsigma_3 B\|r_{I_n}^{\delta_j}\|^{s-1}\|r_{I_n}^{\delta_j}\|^{p-s}$$ $$
        = \varsigma_3 B\|r_{I_n}^{\delta_j}\|^{p-1} 
         \rightarrow \varsigma_3 B \|r_{I_n}\|^{p-1} =0
    $$
    as $j \rightarrow \infty.$ Further, from the definition of $\zeta_{n+1}$, we have
    \begin{align*}
        \zeta_{n+1}^{\delta_j} - \zeta_{n+1} & = (1-\alpha_n)(\zeta_{n}^{\delta_j} - \zeta_n) - \omega_n^{\delta_j}\psi_{I_n}^{\delta_j}.
    \end{align*}
By taking the norm and applying the triangle inequality, we obtain
    \begin{align*}
        \|\zeta_{n+1}^{\delta_j} - \zeta_{n+1}\| &\leq (1-\alpha_n)\|\zeta_{n}^{\delta_j} - \zeta_n\| + \omega_n^{\delta_j}B\|r_{I_n}^{\delta_j}\|^{s-1} \\
        & \leq (1-\alpha_n)\|\zeta_{n}^{\delta_j} - \zeta_n\| + \varsigma_3 B \|r_{I_n}^{\delta_j}\|^{p-1} \rightarrow 0
     \end{align*}
     as $j \rightarrow \infty.$ This implies that $\zeta_{n+1}^{\delta_j} \rightarrow \zeta_{n+1}$ as $j \to \infty$. Furthermore, from the continuity of $\nabla\Psi^*$, we can conclude that $u_{n+1}^{\delta_j} \rightarrow u_{n+1}$ as $j \to \infty$. \\
     \textit{Case-$2$: $r_{I_n} \neq 0$.} Utilizing assumption \ref{assumption c}  along with the characteristics of the mapping $j_{s}^{\mathcal{V}}$, we have
     \begin{equation*}
         \|\psi_{I_n}^{\delta_j}\|_{\mathcal{W}} \leq \hat{B}\|j_s^{\mathcal{V}}(F(u_n^{\delta_j}) - v^{\delta_j})\|_{\mathcal{V}^*} \leq \hat{B}\|F(u_n^{\delta_j}) - v^{\delta_j}\|_{\mathcal{V}}^{s-1}.
     \end{equation*}
     Using  the fact that $u_n^{\delta_j} \rightarrow u_n$ and $v^{\delta_j} \rightarrow v^{\dagger}$ as $j \rightarrow \infty$ and the continuity of $F$, we can see that $\{\psi_{I_n}^{\delta_j}\}$ is bounded in $\mathcal{W}.$ Since $\mathcal{W}$  is compactly embedded in $\mathcal{U}^*$, we can, if necessary, extract a subsequence from $\{v^{\delta_j}\}$ to ensure that $\psi_{I_n}^{\delta_j} \rightarrow \psi_{I_n}$ in $\mathcal{U}^*$ as $j \rightarrow \infty$ for some $\psi_{I_n} \in \mathcal{U}^*.$ We define $\omega_n$ and $(\zeta_{n+1}, u_{n+1})$ by \ref{omega} and \ref{construction}, respectively.  We emphasize that $\omega_n^{\delta_j} \to \omega_n$ does {not} follow merely from the continuity of $F$ and $\nabla \Psi^*$. Since $r_{I_n}^{\delta_j} = F_{I_n}(u_n^{\delta_j}) - v_{I_n}^{\delta_j} \to F_{I_n}(u_n) - v_{I_n}^{\dagger} = r_{I_n}$ by continuity of $F$, and $\psi_{I_n}^{\delta_j} = L_{I_n}(u_n^{\delta_j})^*j_s^{\mathcal{V}_{I_n}}(r_{I_n}^{\delta_j}) \to \psi_{I_n}$ (along the subsequence extracted via the compactness Assumption \ref{assumption c}), we may substitute $\psi_{I_n}^{\delta_j}$ for the  term $L_{I_n}(u_n^{\delta_j})^*j_s^{\mathcal{V}_{I_n}}(r_{I_n}^{\delta_j})$ appearing in the denominator of \eqref{omega}. As every remaining quantity in \eqref{omega} ($r_{I_n}^{\delta_j}$, $t_n^{\delta_j} = \|\zeta_n^{\delta_j}-\zeta_0\|$, and $\varsigma_{2,n}$, which depends continuously on these) converges, passing to the limit in \eqref{omega} along this subsequence yields $\omega_n^{\delta_j} \to \omega_n$, where $\omega_n$ is given by \eqref{omega} with $\psi_{I_n}$ in place of $L_{I_n}(u_n)^*j_s^{\mathcal{V}_{I_n}}(r_{I_n})$. Consequently, from the definition \eqref{construction} and the continuity of $\nabla \Psi^*$, we obtain $\zeta_{n+1}^{\delta_j} \to \zeta_{n+1}$ and $u_{n+1}^{\delta_j} \to u_{n+1}$ as $j \to \infty$. Finally, it remains to verify that \( \psi_{I_n} \) satisfies conditions \ref{psinsatisfy} and \ref{psi 2}, which can be directly concluded from \cite[Lemma 2.3]{rreal2020revisit}.

  Next, from the definition of $\zeta_n^{\delta_j}$ we can see that 
    $$
        \zeta_{n+1}^{\delta_j} - \zeta_0 = \sum_{i = 1}^{n+1} (\zeta_i^{\delta_j} - \zeta_{i-1}^{\delta_j}) 
         \in \text{ran}(L(u_n^{\delta_j})^*)  \oplus \text{ran}(L(u_{n-1}^{\delta_j})^*) \oplus \cdots \oplus  \text{ran}(L(u_{0}^{\delta_j})^*), $$ As $\text{ran}(L(u_n^{\delta_i})^*)\subset \mathcal{N}(L(u_n^{\delta_i}))^\perp$ and $\mathcal{N}(L(u^{\dagger})) \subset \mathcal{N}(L(u))$ for all $u \in \mathcal{B}(u_0, 3\wp),$ we may use $u_n^{\delta_j} \in \mathcal{B}(u_0, 3\wp)$ to note that $\zeta_{n+1}^{\delta_j} - \zeta_0 \in \mathcal{N}(L(u^{\dagger}))^\perp.$ By taking $j \to \infty$, we get $\zeta_{n+1} - \zeta_0 \in \mathcal{N}(L(u^{\dagger}))^\perp.$
\end{proof}

  Finally, we are prepared to establish the main convergence result ans its proof is motivated by \cite{Huang2025early, rreal2020revisit}.
   \begin{theorem}
      Let $\tau >1$ and $\varsigma_1, \varsigma_4 > 0$ be chosen such that $\ref{codn2}$ is satisfied and let assumptions $\ref{assumption a}$, $\ref{assumption b}$ and $\ref{assumption c}$ also hold. Consider a sequence of perturbed data $\{v^{\delta} \}$ satisfying $\|v^{\delta} - v \| \leq \delta \rightarrow 0$ and let the d--SMD method be defined by $\ref{d--SMD 3.1}$ that terminates at the stopping index $n_{\delta}$ defined in $\ref{discrepancy}$. Let $\{v^{\delta_j}\},$ with $\delta_j \rightarrow 0$ as $j \rightarrow \infty,$ be any sequence of $\{v^{\delta}\}.$ Then by taking a subsequence of $\{v^{\delta_j}\}$ if necessary, there holds
      \begin{equation}\label{eqn:main_theorem_1}
          \|u_{n_{\delta_j}}^{\delta_j} - \overline{u}\| \to 0 \hspace{5mm} \text{and} \hspace{5mm} \mathcal{D}_{\zeta_{n_{\delta_j}}^{\delta_j}} \Psi (\overline{u}, u_{n_{\delta_j}}^{\delta_j}) \to 0 \quad \text{almost surely}
      \end{equation}
      and 
       \begin{equation*}
           \mathbb{E}\left[\|u_{n_{\delta_j}}^{\delta_j} - \overline{u}\|^p\right] \rightarrow 0 \hspace{5mm} \text{and} \hspace{5mm} \mathbb{E}\left[\mathcal{D}_{\zeta_{n_{\delta_j}}^{\delta_j}} \Psi (\overline{u}, u_{n_{\delta_j}}^{\delta_j})\right] \rightarrow 0
       \end{equation*}
       as $j \rightarrow \infty$ for some random solution $\overline{u}$ of $\ref{main}$. If, in addition, $\mathcal{N}(L(u^{\dagger})) \subset \mathcal{N}(L(u))$  $\forall \;u \in \mathcal{B}(u_0, 3\wp),$ then 
       \begin{equation*}
           \mathbb{E}\left[\|u_{n_{\delta_j}}^{\delta_j} - u^{\dagger}\|^p\right] \rightarrow 0 \hspace{5mm} \text{and} \hspace{5mm} \mathbb{E}\left[\mathcal{D}_{\zeta_{n_{\delta_j}}^{\delta_j}} \Psi (u^{\dagger}, u_{n_{\delta_j}}^{\delta_j})\right] \rightarrow 0\ \ \text{as } j \to \infty,
       \end{equation*}
       where $u^\dagger$ be the unique solution of \eqref{udagger}.
   \end{theorem}
   \begin{proof}
   Let $\{(\zeta_n,u_n)\}$ denote the random sequence generated by exact data counterpart of Algorithm~\ref{alg:dsmd}. Let $\bar{u}$ be the random solution of~\eqref{main} identified in Theorem~\ref{thm alpha n sumcond}. Then, along every sample path belonging to the event $\Xi$, we have
\[
\mathcal{D}_{\zeta_n} \Psi(\bar{u}, u_n) \to 0
\qquad\text{as } n\to\infty, 
\]
where the event $\Xi$ is defined as 
\[\Xi:= \left\{ \sum_{n=0}^{\infty} \sum_{I : |I|=\ell} \|F_{I}(u_n) - v_{I}^\dagger\|^p < \infty \right\}\] 
and $\mathbb{P}(\Xi)=1$, from \eqref{eqn:for_event}.
Let
\(
\Lambda_n^\delta:=\mathcal{D}_{\zeta_n^\delta}\Psi(\bar u,u_n^\delta)
\) for \( n\geq0.\)
We first establish \eqref{eqn:main_theorem_1}. It is sufficient to prove that
\(
\Lambda_{n_{\delta_j}}^{\delta_j}\to 0
\) as \( j \to \infty
\)
on the event $\Xi$. Fix an arbitrary sample path in $\Xi$ and define
\(
\widehat n:=\liminf_{j \to \infty}n_{\delta_j}
\)
along this path.
We first consider the case $\widehat n<\infty$. Then, for sufficiently large $j$, we have $n_{\delta_j}=\widehat n$. We claim that
\(
u_{\widehat n}=\bar u.
\)
By Lemma~\ref{Lemma 3.1},
\[
\zeta_{\widehat n}^{\delta_{j}}\to \zeta_{\widehat n},
\qquad
u_{\widehat n}^{\delta_{j}}\to u_{\widehat n}
\qquad\text{as } j\to\infty.
\]
Moreover, the stopping criterion gives
\(
\|F_{I_{\widehat n}}(u_{\widehat n}^{\delta_j}) - v_{I_{\widehat n}}^{\delta_j}\|^s \leq (\tau \delta_{j,I_{\widehat n}})^s,
\)
where $\{\delta_{j,I_{\widehat n}}\}$ is the sequence of local noise level  $\delta_{I_{\widehat n}}$ and $\delta_{j,I_{\widehat n}} \to 0 $ as $j \to \infty$.
Passing to the limit as $j\to\infty$ yields
\(
F_{I_{\widehat n}}(u_{\widehat n}) = v_{I_{\widehat n}}^{\dagger}
\)
Hence, according to Algorithm~\ref{alg:dsmd}, the iteration remains unchanged from index
$\widehat n$ onward, i.e.,
\(
\zeta_n=\zeta_{\widehat n}, \) \( u_n=u_{\widehat n},\) for \( n\geq\widehat n.
\)
Since Theorem~\ref{thm alpha n sumcond} guarantees that $u_n \to \bar u$ along the considered sample
path, it follows that
\(
u_{\widehat n}=\bar u.
\)
Consequently, Lemmas~\ref{lemma 3.2} and~\ref{Lemma 3.1}, together with the lower semicontinuity of
$\Psi$, deduce that
\[
\begin{aligned}
\limsup_{j \to \infty}\Lambda_{n_{\delta_j}}^{\delta_j}
&\leq
\limsup_{j \to \infty} \Lambda_{\widehat n}^{\delta_j}\\
&=
\limsup_{j \to \infty}
\left[
\Psi(\bar u)- \Psi(u_{\widehat n}^{\delta_j})
-\left\langle
\zeta_{\widehat n}^{\delta_j},
\bar u - u_{\widehat n}^{\delta_j}
\right\rangle
\right]\\
&=
\Psi(\bar u)- \Psi(u_{\widehat n})
-\left\langle
\zeta_{\widehat n},
\bar u- u_{\widehat n}
\right\rangle =0.
\end{aligned}
\]
Since $\Lambda_{n_{\delta_j}}^{\delta_j}\geq0$, we conclude that
\(
\Lambda_{n_{\delta_j}}^{\delta_j} \to 0 \) as \(j \to \infty.
\)
It remains to consider the case $\widehat n=\infty$. In this case fix any $n\geq0$.
Then $n_{\delta_j}>n$ for  sufficiently large $j$. From Lemma~\ref{lemma 3.2},
\(
\Lambda_{n_{\delta_j}}^{\delta_j} \leq \Lambda_n^{\delta_j}.
\)
Therefore,
\[
\limsup_{j \to \infty}\Lambda_{n_{\delta_j}}^{\delta_j}
\leq
\limsup_{j \to \infty}\Lambda_n^{\delta_j}
=
\limsup_{j \to \infty}
\left[
\Psi(\bar u)-\Psi (u_n^{\delta_j})
-\left\langle
\zeta_n^{\delta_j}, \bar u-u_n^{\delta_j}
\right\rangle
\right].
\]
Using Lemma~\ref{Lemma 3.1} together with the lower semicontinuity of $\Psi$, we obtain
\[
\limsup_{j \to \infty}\Lambda_{n_{\delta_j}}^{\delta_j}
\leq
\Psi(\bar u)-\Psi (u_n)
-\left\langle
\zeta_n,\bar u-u_n
\right\rangle
=
\mathcal{D}_{\zeta_n}\Psi(\bar u,u_n).
\]
This estimate holds for every $n\geq0$. By Theorem~\ref{thm alpha n sumcond},
\(
\mathcal{D}_{\zeta_n}\Psi(\bar u,u_n) \to0
\) as \(n\to\infty.
\)
Letting $n\to\infty$ consequently gives
\(\limsup_{j \to \infty}\Lambda_{n_{\delta_j}}^{\delta_j}\leq 0.
\)
Since the Bregman distance is nonnegative, we again conclude that
\(
\Lambda_{n_{\delta_j}}^{\delta_j} \to0 \)
as \( j \to\infty.\)
Thus, $\Lambda_{n_{\delta_j}}^{\delta_j} \to0$ along every sample path in $\Xi$,
and hence
\[
\Lambda_{n_{\delta_j}}^{\delta_j} \to0
\qquad\text{almost surely as }  j \to\infty.
\]
Furthermore,
\(
0\leq\Lambda_{n_{\delta_j}}^{\delta_j}
\leq\Lambda_0^{\delta_j}
=\mathcal{D}_{\zeta_0}\Psi(\bar u,u_0),
\)
which is independent of noise level. Therefore, by the dominated convergence
theorem,
\(
\mathbb{E}\left[\Lambda_{n_{\delta_j}}^{\delta_j}\right]\to0
\) as \( j\to \infty.
\)
Finally, the $p$-convexity of $\Psi$ yields
\[
\mathbb{E}\left[
\|u_{n_{\delta_j}}^{\delta_j}- \bar u\|^p
\right]\to0,
\qquad j \to \infty,
\]
which proves the result. The subsequent result can be followed by \cite[Theorem 2.6]{rreal2020revisit}.
    \end{proof} 
\section{Numerical experiments and discussion}\label{sec:numerics}

In this section we test d--SMD Algorithm~\ref{alg:dsmd} on the following two applications: a system of non-smooth semilinear inverse source problems and the identification of a diffusion coefficient.
For every experiment we report
\begin{itemize}
    \item the stopping index $n_\delta$ determined by the discrepancy principle \eqref{discrepancy}, together with the number of iterations $n_{\mathrm{upd}}$ for which $\omega_n^{\delta}>0$, corresponding to iterations where the iterate is updated,
    \item the relative reconstruction error,
    \(
    e_{\mathrm{rel}}=\frac{\|u_{n_\delta}^{\delta}-u^{\dagger}\|}{\|u^{\dagger}\|},
    \)
    \item the final residual,
    \(
    \|F_{I_{n_\delta}}(u_{n_\delta}^{\delta})-v_{I_{n_\delta}}^{\delta}\|,
    \)
    \item the peak signal-to-noise ratio (PSNR) and structural similarity index measure (SSIM) of the reconstructed image for Example~\ref{ex:diffusion}, where the unknown is a two-dimensional image,
    \item the computational cost, measured by the number of sparse linear solves ($\#\mathcal{S}$) or matrix--vector products ($\#\mathcal{M}$), which serve as machine-independent measures of computational effort, together with the corresponding CPU time.
\end{itemize}
Since the block index $I_n$ is selected randomly, both the iterates $u_n^{\delta}$ and the stopping index $n_\delta$ are random variables. Therefore, the results reported for the stochastic methods are obtained by averaging over multiple independent realizations of the block sampling, with the corresponding standard deviations given in parentheses. In contrast, the deterministic methods are executed only once. Furthermore, in Subsection~\ref{sec:assumptions-numeric}, we numerically verify that Assumptions~\ref{assumption a}--\ref{assumption c}, as well as the conditions on the step size \eqref{step} and the damping factor \eqref{condonalpha}, are satisfied for both the problems.
\subsection{Common set-up}\label{sec:setup}
Throughout the numerical experiments, we consider the Hilbert space setting with
$p=s=2$, so that $p^*=2$ and $j_s^{\mathcal{V}_{I_n}}=\mathrm{Id}$, unless stated otherwise. We choose
\[
\varsigma_1=1,\qquad
\rho=0.9,\qquad
\varsigma_{2,n}
=(1-\rho)\,\varsigma_1^{p^*-1}
\frac{\|r_{I_n}^{\delta}\|^{s}}{(t_n^{\delta})^{p^*}},
\qquad
\overline{\varsigma}_2
=(1-\rho)^{1/(p^*-1)}\varsigma_1,
\]
which corresponds to the constructive choice proposed after \eqref{step}. The parameter $\varsigma_3$ is set to $10^{8}$ for Example~\ref{ex:source} and to $10^{10}$ for Examples~\ref{ex:diffusion}. These values are chosen sufficiently large so that the upper bound $\varsigma_3\|r_{I_n}^{\delta}\|^{p-s}$ is never active.

The damping parameter is defined as
\[
\alpha_n
=\theta\,
\min\Big\{
\varsigma_4\,\omega_n^{\delta}
\|F_{I_n}(u_n^{\delta})-v_{I_n}^{\delta}\|^{s}(t_n^{\delta})^{-1},\;
2^{\frac{1-p^*}{p^*}}
(\omega_n^{\delta}\varsigma_{2,n})^{\frac{1}{p^*}},\;
\varsigma_6\omega_n^{\delta}
\Big\},
\qquad
\varsigma_4=\varsigma_6=1,
\]
where the scaling parameter $\theta\in[0,1]$ ensures that $\alpha_n$ satisfies the admissibility condition \eqref{condonalpha}. The special choice $\theta=0$ recovers the undamped stochastic mirror descent method. Unless stated otherwise, we use $\theta=10^{-2}$ throughout the experiments. The influence of this parameter on the algorithm's performance is investigated in Table~\ref{tab:ex1damp}.
\vspace{0.2cm}
\paragraph{\textbf{Choice of the penalty functional $\Psi$.}}
We consider three penalty functionals, each satisfying Assumption~\ref{assumption a} with $p=2$ and $c_0=1/\gamma$:
\begin{align}
 \Psi_{\mathrm{quad}}(u)
 &= \tfrac{1}{\gamma}\|u\|_{L^2}^2+\mathbb{I}_{\mathcal{C}}(u),
 &&\mathcal{C}=\{u:\ \underline{u}\le u\le \overline{u}\},
 \label{eq:psiquad}\\
 \Psi_{\mathrm{TV}}(u)
 &= \tfrac{1}{\gamma}\|u\|_{L^2}^2+\beta\,\mathrm{TV}(u)+\mathbb{I}_{\mathcal{C}}(u),
 \label{eq:psitv}\\
 \Psi_{\Theta}(u)
 &= \tfrac{1}{\gamma}\|u\|_{L^2}^2+\beta\,\|\Theta u\|_{\ell^1}+\mathbb{I}_{\mathcal{C}}(u).
 \label{eq:psitheta}
\end{align}

The first two functionals correspond to the standard quadratic and total variation (TV) regularization, respectively. The third functional incorporates a learned analysis operator $\Theta$ and represents the trained convex penalty. Here, $\Theta$ consists of a bank of $K=6$ convolution filters of size $3\times3$, learned from the training pairs $\{u^{(l)},v_i^{(l)}\}$ using analysis-operator learning. The filters are obtained as the least energetic principal directions of $3\times3$ image patches from the training data, thereby promoting sparse representations for images with similar characteristics.

This learned penalty provides a convex alternative to the RU-net training strategy proposed in \cite[Algorithm~3]{Long2024}. An important advantage of this formulation is that the mirror update
\[
u_{n+1}^{\delta}
=\arg\min_u
\left\{
\Psi(u)-\langle\zeta_{n+1}^{\delta},u\rangle
\right\}
\]
remains a convex proximal problem, which is solved accurately using FISTA on the dual formulation. The learned convolution filters are shown in Figure~\ref{fig:filters}.
For the quadratic penalty $\Psi_{\mathrm{quad}}$, the mirror map reduces to the metric projection
\(
u=P_{\mathcal{C}}\!\left(\frac{\gamma}{2}\zeta\right).
\)
For the TV penalty $\Psi_{\mathrm{TV}}$, the mirror map is solved using Chambolle's dual algorithm accelerated by FISTA with warm initialization~\cite{beck2009fast}.
Unless stated otherwise, we set
\(
\gamma=\frac{4}{\rho\varsigma_1},
\)
which normalizes the effective primal step size of Algorithm~\ref{alg:dsmd} to
\(
\frac{\|r_{I_n}^{\delta}\|^{s}}
{\|L_{I_n}(u_n^{\delta})^{*}j_s(r_{I_n}^{\delta})\|^{p}}.
\)
\begin{figure}[htbp]\centering
\includegraphics[width=.95\textwidth]{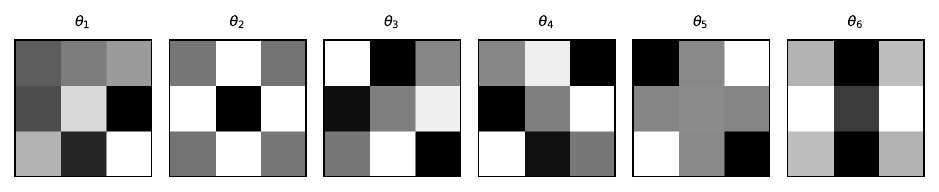}
\caption{The $K=6$ learned analysis filters $\theta_1,\dots,\theta_6$ defining the trained
convex penalty $\Psi_\Theta$ in \eqref{eq:psitheta}.}
\label{fig:filters}
\end{figure}

\paragraph{\textbf{Competing methods.}}
We compare Algorithm~\ref{alg:dsmd} with the following representative methods:

\begin{enumerate}
    \item[(i)] the \emph{Bouligand--Landweber} method proposed in \cite{clason2019bouligand}, implemented in its Kaczmarz variant with a constant step size $\omega\equiv 1.8/B^2$ and cyclic updates over the blocks $i=1,\ldots,M$,

    \item[(ii)] the \emph{regularized block Kaczmarz method} (RBKM)~\cite{Bao2016}. The regularization parameter $\gamma$ is selected by an exhaustive grid search over $\{10^{-1},\ldots,10^{-12}\}$, choosing the value that yields the smallest terminal relative error. This tuning strategy provides a favorable setting for the RBKM baseline,

    \item[(iii)] the \emph{undamped} stochastic mirror descent method, obtained from Algorithm~\ref{alg:dsmd} by setting $\alpha_n\equiv 0$. This serves as an ablation study to assess the effect of the proposed damping strategy,

    \item[(iv)] the \emph{iteratively regularized Gauss--Newton method} (IRGNM), applied to the full system with regularization parameters of the form $\phi_n=\phi_0q^n$ and solved using the conjugate gradient (CG) method as the inner solver.
\end{enumerate}
To ensure a fair comparison, all methods are initialized with the same starting point $u_0$ and terminated using the same discrepancy principle with the parameter $\tau$. 

\paragraph{\textbf{Implementation of the stopping rule.}}
The discrepancy principle \eqref{discrepancy} is evaluated only on the sampled block $I_n$. Accordingly, the step size defined in \eqref{step} satisfies $\omega_n^{\delta}=0$ whenever
\(
\|r_{I_n}^{\delta}\|^{s}\leq (\tau\delta_{I_n})^{s}.
\)
Thus, if the sampled block already satisfies the discrepancy principle, no update is performed. Consequently, the iteration should terminate only when the discrepancy principle is satisfied for every block $I \subset \{1,\ldots,M\}$. Otherwise, the algorithm may terminate prematurely simply because a randomly selected block satisfies the stopping criterion, even though other blocks do not. This issue is particularly significant when the individual blocks are underdetermined, as in Example~\ref{ex:helmholtz}, where each block consists of only $64$ equations for $1024$ unknowns.

Sampling the block index uniformly from $\{1,\ldots,M\}$ throughout the iteration is inefficient, since blocks that already satisfy the discrepancy principle are repeatedly selected but do not contribute to further updates. To avoid this unnecessary computation, we maintain an \emph{active set}
\[
\mathcal{A}_n=
\left\{
i:\,
\|F_i(u_n^{\delta})-v_i^{\delta}\|^{s}
>
(\tau\delta_i)^{s}
\ \text{based on the available information}
\right\},
\]
from which the block index $I_n$ is sampled. Whenever a sampled block is found to satisfy the discrepancy principle, it is removed from the active set. If the active set becomes empty, a final verification sweep over all $M$ blocks is performed. If every block satisfies the discrepancy principle, the algorithm terminates. Otherwise, the blocks that violate the stopping criterion are added back to the active set, and the iteration continues.
\begin{rema}
    The active-set strategy is introduced solely to improve computational efficiency. It does not change the sequence of actual updates, but only avoids iterations in which no update would be performed. Therefore, the theoretical analysis remains unchanged, and the reconstructed solution is the same up to the randomness introduced by block sampling.
\end{rema}
Unless stated otherwise, all numerical results reported in this section are obtained using the active-set strategy. 

\subsection{Example 1: System of non-smooth semilinear inverse source problems}\label{ex:source}
In this section we test the d--SMD method \ref{d--SMD 3.1} ona system of non-smooth semilinear inverse source problems.


Let $\mathcal{I}\subset\mathbb{R}^2$ be open and bounded domain with Lipschitz boundary $\partial\mathcal{I}$.  For
$i=1,\ldots,M$ let $\mathcal{F}_i$ denote the solution operator,
i.e.\ $\mathcal{F}_i(u)=v_i$ where $v_i$ solves
\begin{equation}\label{eq:ex1state}
  -\Delta v_i + i\,v_i^{+} = u \ \ \text{in }\mathcal{I},
  \qquad v_i = 0 \ \ \text{on }\partial\mathcal{I},
  \qquad v_i^{+}(t)=\max\{v_i(t),0\}.
\end{equation}
For $i=1$, problem \eqref{eq:ex1state} coincides with the model problem studied by Clason and Nhu~\cite{clason2019bouligand}. The parameter $i>0$ simply scales the monotone nonsmooth term and does not affect the analytical properties of the problem. Consequently, the results established in \cite{clason2019bouligand} extend directly to our setting. Below, we summarize the key properties required to show that problem \eqref{eq:ex1state} fits into the framework developed in Algorithm~\ref{alg:dsmd}.

\begin{enumerate}
\item[(i)] \emph{Well-posedness and Lipschitz continuity.}  Since $t\mapsto i\,t^{+}$ is
monotone and globally Lipschitz, \eqref{eq:ex1state} has for every $u\in L^2(\mathcal{I})$
a unique solution, and
\[
  \mathcal{F}_i:\ L^2(\mathcal{I})\longrightarrow H_0^1(\mathcal{I})\cap C(\overline{\mathcal{I}})
\]
is globally Lipschitz continuous \cite[Proposition 2.1]{christof2017optimal}.  Composed
with the compact embedding $H_0^1(\mathcal{I})\hookrightarrow L^2(\mathcal{I})$
it is moreover completely continuous as an operator
$L^2(\mathcal{I})\to L^2(\mathcal{I})$ \cite[Lemma 3.2]{clason2019bouligand}. In
particular $\mathcal{F}_i$ is weakly closed, which is Assumption \ref{assumption b}(i).

\item[(ii)] \emph{Directional but not G\^ateaux differentiability.}  For every
$u,h\in L^2(\mathcal{I})$ the directional derivative $\mathcal{F}_i'(u;h)$ exists and is
the unique solution of an elliptic variational inequality
\cite[Theorem 2.2]{christof2017optimal}. The map $h\mapsto\mathcal{F}_i'(u;h)$ is however
only positively homogeneous, and $\mathcal{F}_i$ fails to be G\^ateaux differentiable at
every $u$ for which the biactive set $\{v_i=0\}$ has positive measure
\cite[Proposition 3.4]{clason2019bouligand}.  Since the Fr\'echet derivative of \eqref{eq:ex1state} is not available everywhere, an SGD-type method~\cite{GU2024,jin2026} cannot be formulated using the classical derivative. This is precisely the motivation for introducing the family of operators $\{L_{I_n}(u)\}$ in Assumption~\ref{assumption b}(iii), which provides a suitable generalized linearization for such nonsmooth problems.

\item[(iii)] \emph{The Bouligand subdifferential.}  Let
\begin{equation*}
    \begin{split}
          \partial_B\mathcal{F}_i(u)
  =\Big\{\,G_i \in\mathbb{B}\big(L^2(\mathcal{I})\big):\
     G_i=\lim_{k\to\infty}\mathcal{F}_i'(u_k),\ u_k\to u \Big\}
    \end{split}
\end{equation*}
be the Bouligand subdifferential, where $ \mathcal{F}_i$ is G\^ateaux differentiable at  $u_k$.  By \cite[Proposition 3.16]{christof2017optimal} every
$G_i\in\partial_B\mathcal{F}_i(u)$ is the solution operator of a linear elliptic
equation with a characteristic-function coefficient $Gh=w$, where
\begin{equation}\label{eq:ex1bouligand}
  -\Delta w + i\,\chi_{A}\,w = h \ \ \text{in }\mathcal{I},
  \qquad w=0 \ \ \text{on }\partial\mathcal{I},
\end{equation}
for some measurable $A$ with
$\{v_i>0\}\subset A\subset\{v_i\ge0\}$ up to a set of measure zero.  
Consequently
\[
  L_i(u):=\big(-\Delta+i\,\chi_{A}\big)^{-1}\in\partial_B\mathcal{F}_i(u),
\]
and, $-\Delta+i\chi_A$ being symmetric and coercive, $L_i(u)$ is self-adjoint and
positive on $L^2(\mathcal{I})$ with
\[
  \|L_i(u)\|_{\mathbb{B}(L^2(\mathcal{I}))}\le\frac1{\lambda_1},
  \qquad
  \|L_i(u)\|_{\mathbb{B}(L^2(\mathcal{I}),H_0^1(\mathcal{I}))}\le\hat B,
\]
uniformly in $u$ and $i$, where $\lambda_1$ is the first Dirichlet eigenvalue of
$-\Delta$.  The first bound is the constant $B$ of Assumption \ref{assumption b}(iii). The
second, together with $H_0^1\hookrightarrow L^2$, gives Assumption
\ref{assumption c} with $\mathcal{W}=H_0^1(\mathcal{I})$, since
$\mathrm{range}\,L_i(u)^*=\mathrm{range}\,L_i(u)\subset H_0^1(\mathcal{I})$.  Note that
$\partial_B\mathcal{F}_i(u)$ is a singleton precisely when the biactive set is null, so
\eqref{eq:ex1bouligand} genuinely is a set-valued object.

\item[(iv)] \emph{Generalised tangential cone condition.}  The decisive structural result
of \cite{clason2019bouligand} is that, although $\mathcal{F}_i$ is not differentiable, the
Bouligand subderivative still satisfies a tangential-cone-type estimate, i.e, there is
$\eta<1$ with
\[
  \big\|\mathcal{F}_i(\overline u)-\mathcal{F}_i(u)-L_i(u)(\overline u-u)\big\|
  \ \le\ \eta\,\big\|\mathcal{F}_i(\overline u)-\mathcal{F}_i(u)\big\|
\]
for all $\overline u,u$ in a ball around $u^{\dagger}$
\cite[Proposition 3.16]{clason2019bouligand}.  Applied blockwise this is precisely the
(GTCC) \eqref{tangential} of Assumption \ref{assumption b}(iii). 
\end{enumerate}

Thus the system \eqref{eq:ex1state} satisfies Assumptions
\ref{assumption b} and \ref{assumption c}, and any of the penalties
\eqref{eq:psiquad}--\eqref{eq:psitheta} supplies Assumption \ref{assumption a}.  The
empirical values of $B$, $\eta$ and of the compactness indicators are collected in
Table \ref{tab:assump}.

We discretise \eqref{eq:ex1state} on $\mathcal{I}=(0,1)^2$ by the standard five-point
Laplacian on a uniform grid with $47\times47=2209$ interior nodes, so that
$\mathcal{U}=L^2(\mathcal{I})$ is replaced by $\mathbb{R}^{2209}$, and take $M=10$.  For
each $i$ the state $v_i=\mathcal{F}_i(u)$ is computed by a semismooth Newton iteration
with $A_k=\{v_i^{(k)}>0\}$ the current active set, the step solves
\[
  \big(-\Delta+i\,\chi_{A_k}\big)\,\big(v_i^{(k+1)}-v_i^{(k)}\big)
  = -\big(-\Delta v_i^{(k)}+i\,(v_i^{(k)})^{+}-u\big),
\]
which converges locally superlinearly and, in our runs, in four to six steps from
$v_i^{(0)}=0$.  The matrix appearing here is exactly the discrete counterpart of
\eqref{eq:ex1bouligand} at the converged active set, so the sparse factorisation produced
by the last Newton step is $L_i(u)$.  This is why
the work counter $\#\mathcal{S}$ below counts sparse solves.

The exact source $u^{\dagger}$ is the sum of a disc of height $1.5$, a smooth Gaussian bump
of height $1$ and a rectangle of height $0.8$ (Figure \ref{fig:ex1recon}, first panel). It
is neither smooth nor sparse and possesses jumps, which makes the comparison of the three
penalties meaningful. Data are perturbed by
$v_i^{\delta}=v_i^{\dagger}+\varrho_i\xi_i/\|\xi_i\|$ with $\xi_i$ standard Gaussian and
$\varrho_i=\varrho_{\mathrm{rel}}\|v_i^{\dagger}\|$, so that $\varrho_i$ is exactly the
noise level required in \eqref{delta In def}. We take $u_0=0$, $\zeta_0=0\in\partial\Psi(u_0)$,
$\mathcal{C}=[0,2]$, $\ell=1$, $\tau=1.05$ and, for the feature-based penalties,
$\beta=0.02$.
Table \ref{tab:ex1main} collects the results for three noise levels and Figure
\ref{fig:ex1recon} shows the reconstructions at $\varrho_{\mathrm{rel}}=1\%$. Three key observations can be drawn from the numerical results.

\begin{enumerate}
    \item[(a)] \emph{The adaptive step size improves computational efficiency.} At a relative noise level of $\varrho_{\mathrm{rel}}=0.5\%$, the proposed method achieves a relative error of $e_{\mathrm{rel}}=0.194$ using only $2120$ sparse linear solves. In contrast, the constant-step Bouligand--Landweber method~\cite{clason2019bouligand} does not satisfy the discrepancy principle even after $4000$ Kaczmarz iterations ($15997$ sparse solves) and attains a larger relative error of $0.237$. The IRGNM achieves a similar reconstruction accuracy ($0.202$), but requires $6370$ sparse solves because of the additional conjugate gradient iterations. At $\varrho_{\mathrm{rel}}=1\%$, the proposed method requires only $849$ sparse solves, compared with $14817$ and $5250$ for the Bouligand--Landweber method and IRGNM, respectively. Overall, the proposed algorithm reduces the computational effort while maintaining comparable or better reconstruction accuracy.

    \item[(b)] \emph{Feature-preserving penalties improve reconstruction quality.} The TV-based penalty \eqref{eq:psitv} provides the most accurate reconstructions at noise levels of $\varrho_{\mathrm{rel}}=1\%$ and $3\%$, and more effectively preserves the discontinuity of the disc, as illustrated in Figure~\ref{fig:ex1recon}. At the lowest noise level ($\varrho_{\mathrm{rel}}=0.5\%$), the discrepancy principle permits a larger number of iterations, allowing all three penalty functionals to achieve similar reconstruction errors within one standard deviation.

    \item[(c)] \emph{The damping strategy reduces the number of iterations.} Comparing the undamped stochastic mirror descent method (SMD), corresponding to $\alpha_n\equiv0$, with the proposed method using the quadratic penalty $\Psi_{\mathrm{quad}}$, both approaches produce nearly identical relative errors at all three noise levels, with differences lying within one standard deviation. However, the proposed damped method satisfies the discrepancy principle in fewer iterations. For example, at $\varrho_{\mathrm{rel}}=1\%$, it requires $211$ iterations compared with $256$ for the undamped method, corresponding to an $18\%$ reduction in computational effort. Furthermore, Table~\ref{tab:ex1damp} shows the influence of the scaling parameter $\theta$ in \eqref{condonalpha}. Moderate values, $\theta\in[10^{-3},10^{-2}]$, reduce both the stopping index and the reconstruction error compared with the undamped case ($\theta=0$). In contrast, larger values ($\theta\gtrsim10^{-1}$) deteriorate the reconstruction, as the damping term exerts an excessively strong pull towards the initial dual variable $\zeta_0$. This behavior is consistent with the expected properties of iteratively regularized methods.
\end{enumerate}

Figure~\ref{fig:ex1bl} compares the proposed method with the Bouligand--Landweber method of~\cite{clason2019bouligand}, together with the corresponding pointwise error and a horizontal profile across the discontinuous disc. Both methods produce nearly identical reconstructions, as they are terminated by the same discrepancy principle at the same noise level. In both cases, the reconstruction error is mainly concentrated around the discontinuity, where the quadratic penalty smooths the edge. The main advantage of the proposed method is its computational efficiency, requiring only $849$ sparse linear solves compared to $14817$ for the Bouligand--Landweber method.
Figure~\ref{fig:ex1conv} illustrates the evolution of the relative error, residual, adaptive step size $\omega_n^{\delta}$, and damping parameter $\alpha_n$ throughout the iteration. The stopping index determined by the discrepancy principle \eqref{discrepancy} is also indicated. Both $\omega_n^{\delta}$ and $\alpha_n$ vary by several orders of magnitude during the iteration, highlighting the importance of adaptive parameter selection and explaining the poor performance of a constant step size.
Table~\ref{tab:ex1ell} summarizes the effect of the block size $\ell=|I_n|$. As $\ell$ increases, the number of iterations decreases only slightly, from $216$ for $\ell=1$ to $156$ for $\ell=10$. However, since the computational cost per iteration grows approximately linearly with $\ell$, the total computational effort increases by a factor of about $6.6$ from $\ell=1$ to $\ell=M=10$, without improving the reconstruction accuracy. These results indicate that single-block sampling is the most computationally efficient choice, demonstrating the advantage of the proposed randomized block selection strategy.

\begin{figure}[htbp]\centering
\includegraphics[width=0.7\textwidth]{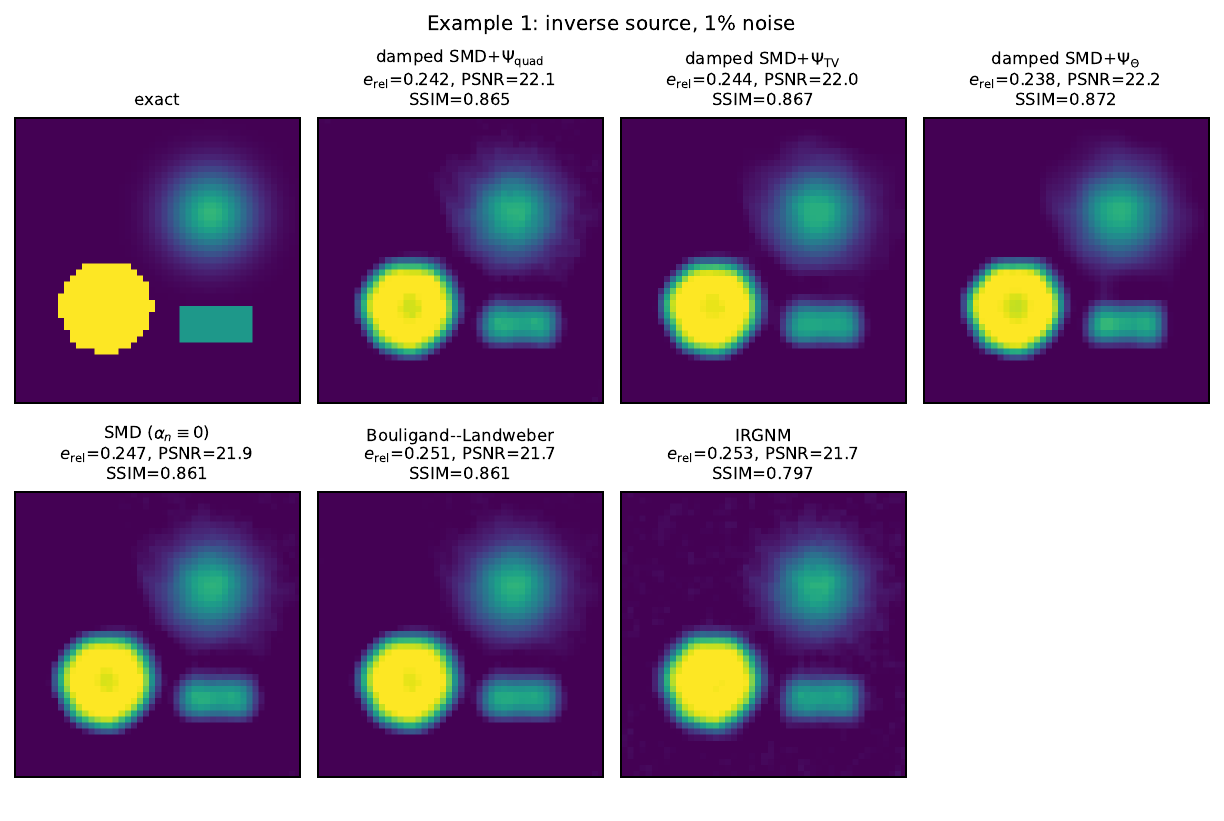}
\caption{Example \ref{ex:source}, $\varrho_{\mathrm{rel}}=1\%$: exact source and the
reconstructions produced at the stopping index $n_\delta$ of \eqref{discrepancy}.}
\label{fig:ex1recon}
\end{figure}

\begin{figure}[htbp]\centering
\includegraphics[width=0.8\textwidth]{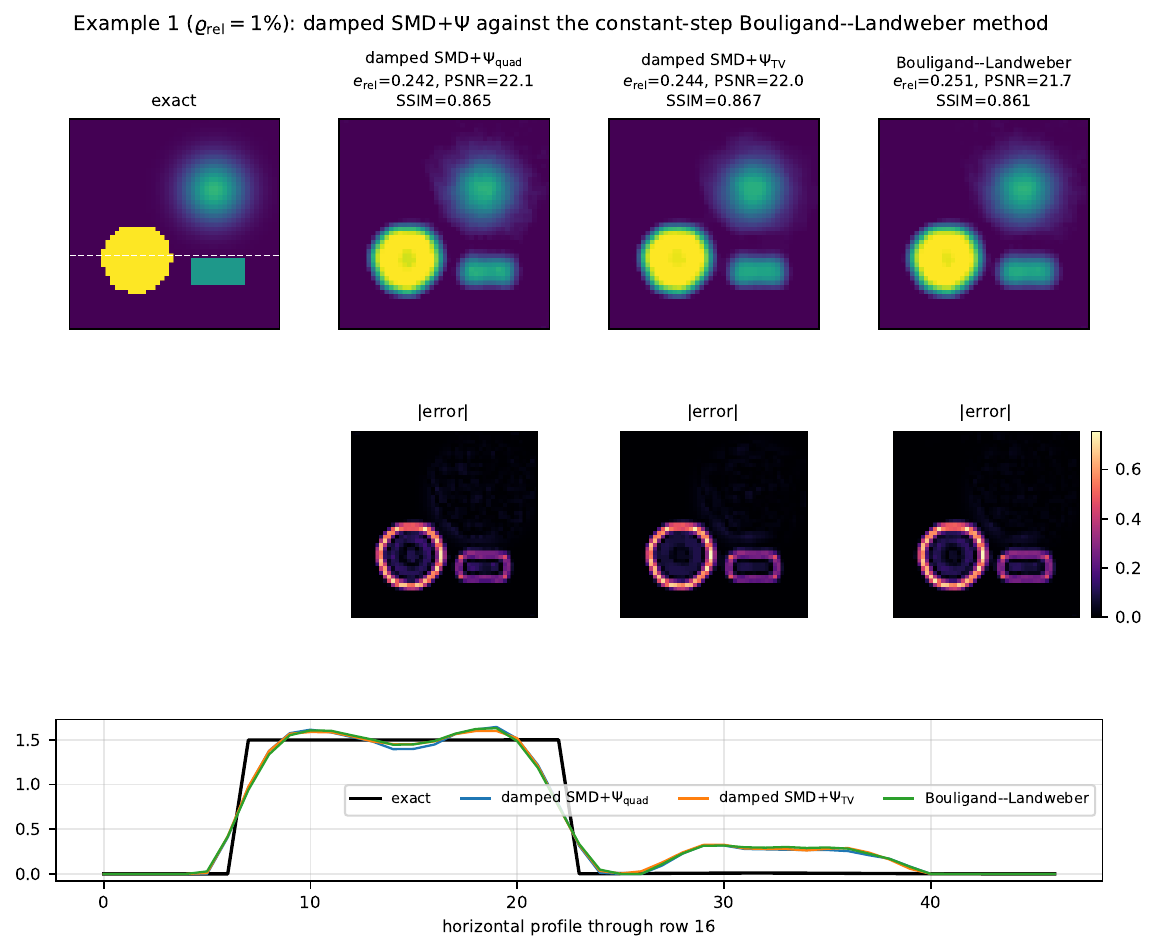}
\caption{Example \ref{ex:source}, $\varrho_{\mathrm{rel}}=1\%$: damped-SMD~\ref{d--SMD 3.1} against
the constant-step Bouligand--Landweber method of \cite{clason2019bouligand}.  Top:
reconstructions at $n_\delta$; middle: pointwise error on a common scale; bottom:
horizontal profile through the disc (dashed line in the first panel).}
\label{fig:ex1bl}
\end{figure}

\begin{figure}[htbp]\centering
\includegraphics[width=0.7\textwidth]{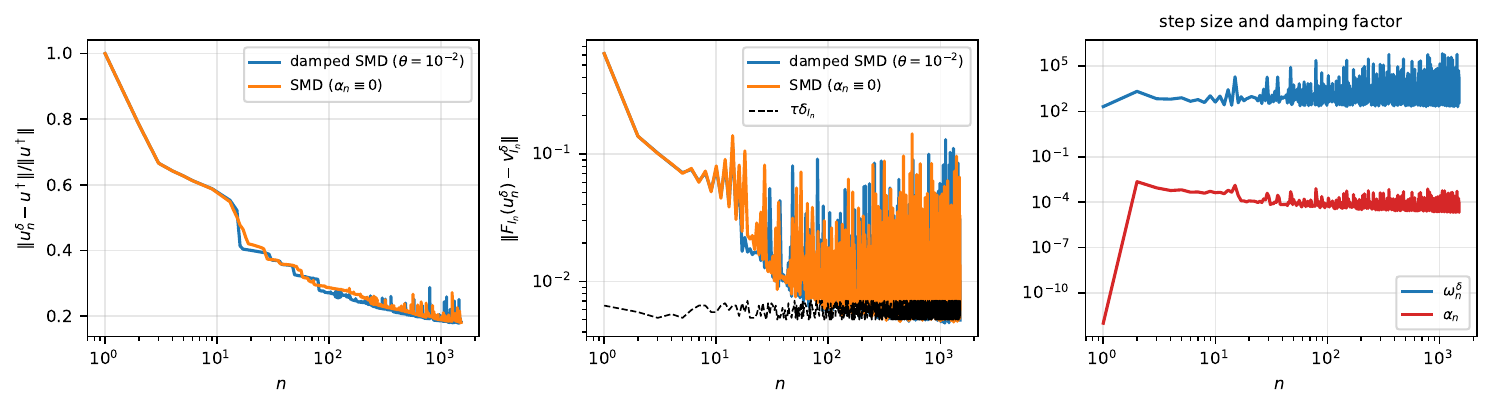}
\caption{Example \ref{ex:source}, $\varrho_{\mathrm{rel}}=1\%$, $\ell=2$. Left: relative
error (the markers indicate $n_\delta$). Middle: residual on the sampled block together
with the threshold $\tau\delta_{I_n}$. Right: realised step size $\omega_n^{\delta}$ from
\eqref{step} and damping factor $\alpha_n$ from \eqref{condonalpha}.}
\label{fig:ex1conv}
\end{figure}




\label{ex:helmholtz}

\subsection{Example 2: Identification of a diffusion coefficient}\label{ex:diffusion}
The previous example is mildly nonlinear and severely ill-posed, as the measurements are available only on the boundary, resulting in data that lie on a lower-dimensional manifold. To further demonstrate the performance of d--SMD method~\ref{d--SMD 3.1} on a nonlinear inverse problem with a unique solution determined by the available data, we consider the classical parameter identification problem
\begin{equation}\label{eq:diffusion}
  -\nabla\!\cdot\!\big(\sigma\nabla u_i\big)=f_i \quad\text{in }\mathcal{I}=(0,1)^2,
  \qquad u_i=0 \quad\text{on }\partial\mathcal{I}, \qquad i=1,\ldots,M,
\end{equation}
with the $M$ source terms $f_i$ playing the role of the blocks and the observation
operators
\[
  F_i(\sigma)=\mathcal{O}\,u_i(\sigma),\qquad u_i(\sigma)=S(\sigma)^{-1}f_i,
  \qquad S(\sigma)=-\nabla\!\cdot\!\sigma\nabla .
\]
Writing the discrete stiffness matrix in the factored form
$S(\sigma)=G^{\!\top}\mathrm{diag}(A\sigma)G$, with $G$ the scaled difference operator and
$A$ the face-averaging operator, the derivative and its adjoint are exact and cheap,
\[
  F_i'(\sigma)h=-\mathcal{O}\,S(\sigma)^{-1}G^{\!\top}\mathrm{diag}(Ah)G\,u_i,
  \qquad
  F_i'(\sigma)^*w=-A^{\!\top}\big[(Gu_i)\odot(Gy)\big],\quad y=S(\sigma)^{-1}\mathcal{O}^{*}w,
\]
and both are verified against finite differences in the accompanying code.  As in the boundary-data problems the operator $S(\sigma)$ does not depend on $i$, so one sparse
factorisation per iterate serves the whole block.

We use $N=64^2$ unknowns, $M=20$ smooth random source terms (low-pass filtered white
noise, which makes $\nabla u_i$ non-degenerate and of varying direction throughout
$\mathcal{I}$, hence $\sigma$ identifiable), full-field data, $\ell=2$, $\tau=1.05$,
$\mathcal{C}=[1,2.5]$, $\beta=0.1$ and the constant initial guess
$\sigma_0\equiv1.5$, for which $e_{\mathrm{rel}}=0.366$.  Two test images are used: the
modified Shepp--Logan phantom and a disc/square/ring image, both rescaled to
$[1,2.5]$.  The empirical tangential cone constant is $\eta_{\max}\approx4\cdot10^{-2}$
(Table \ref{tab:assump}), comfortably inside $\eta<1$.

Tables~\ref{tab:ex2main} and~\ref{tab:ex2shapes}, together with Figures~\ref{fig:ex3recon}--\ref{fig:ex3bl}, present the numerical results. Since the measurements are available throughout the domain rather than only on the boundary, the reconstructed images are considerably sharper. For the Shepp--Logan phantom, Algorithm~\ref{alg:dsmd} with the TV penalty achieves a relative error of $e_{\mathrm{rel}}=0.060$, a PSNR of $26.2$\,dB, and an SSIM of $0.945$. For the second test image, it attains a relative error of $0.042$, a PSNR of $28.6$\,dB, and an SSIM of $0.978$.

This is also the example in which the adaptive step size \eqref{step} makes the largest
difference.  The block operator norms $\|L_i(\sigma)\|$ vary considerably with $i$, so
the constant step $1.8/B^2$ of Landweber--Kaczmarz is dictated by the worst block. After
$6000$ sweeps and $11999$ sparse solves it is still at $e_{\mathrm{rel}}=0.219$ and
$\mathrm{SSIM}=0.436$, i.e.\ it has not converged, while \ref{d--SMD 3.1} terminates
after $920$ draws and $3698$ solves. Figure \ref{fig:ex3bl} shows the three reconstructions, the corresponding error
maps and a horizontal profile through the phantom.

Table \ref{tab:ex2ell} confirms once more that small blocks are the most economical:
$\ell=1$ gives both the smallest error and the smallest work, and the fully deterministic
choice $\ell=M=20$ costs five times as many solves for a slightly worse reconstruction.

\begin{figure}[htbp]\centering
\includegraphics[width=0.7\textwidth]{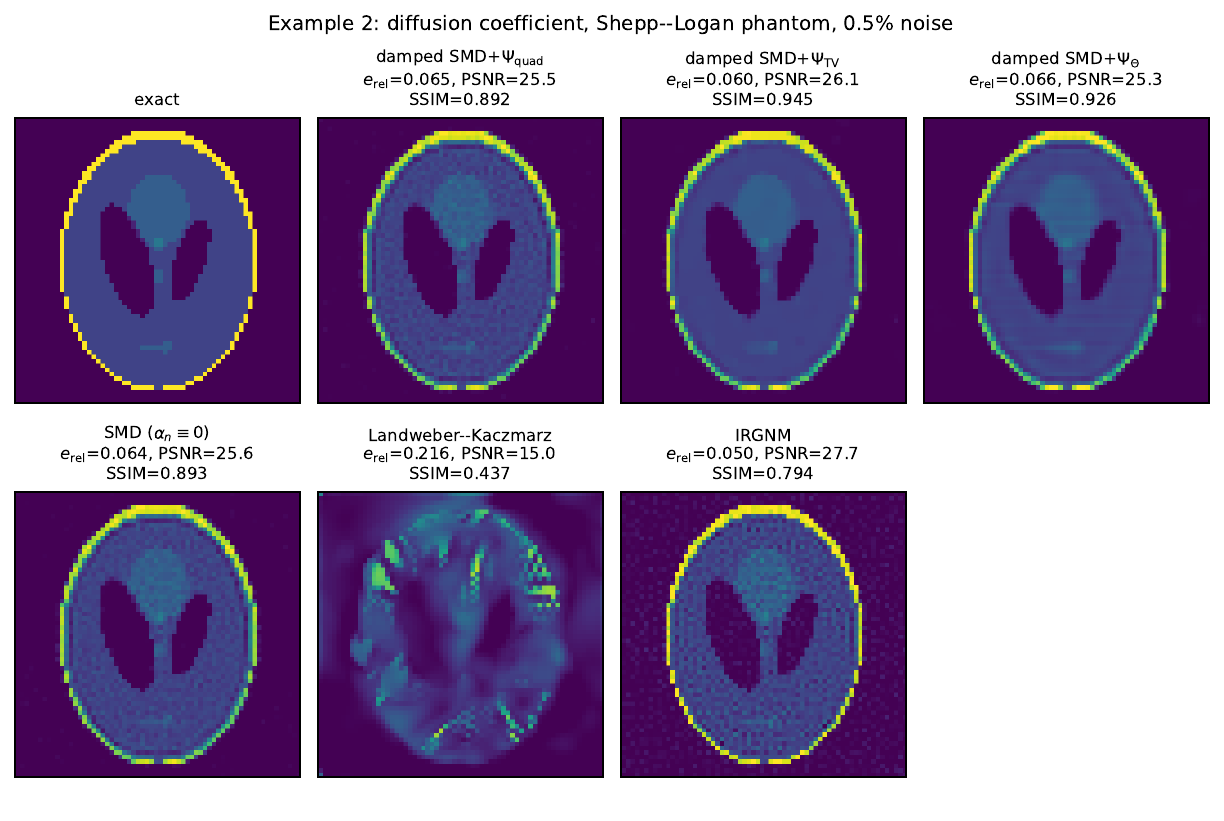}\\[.6em]
\includegraphics[width=0.7\textwidth]{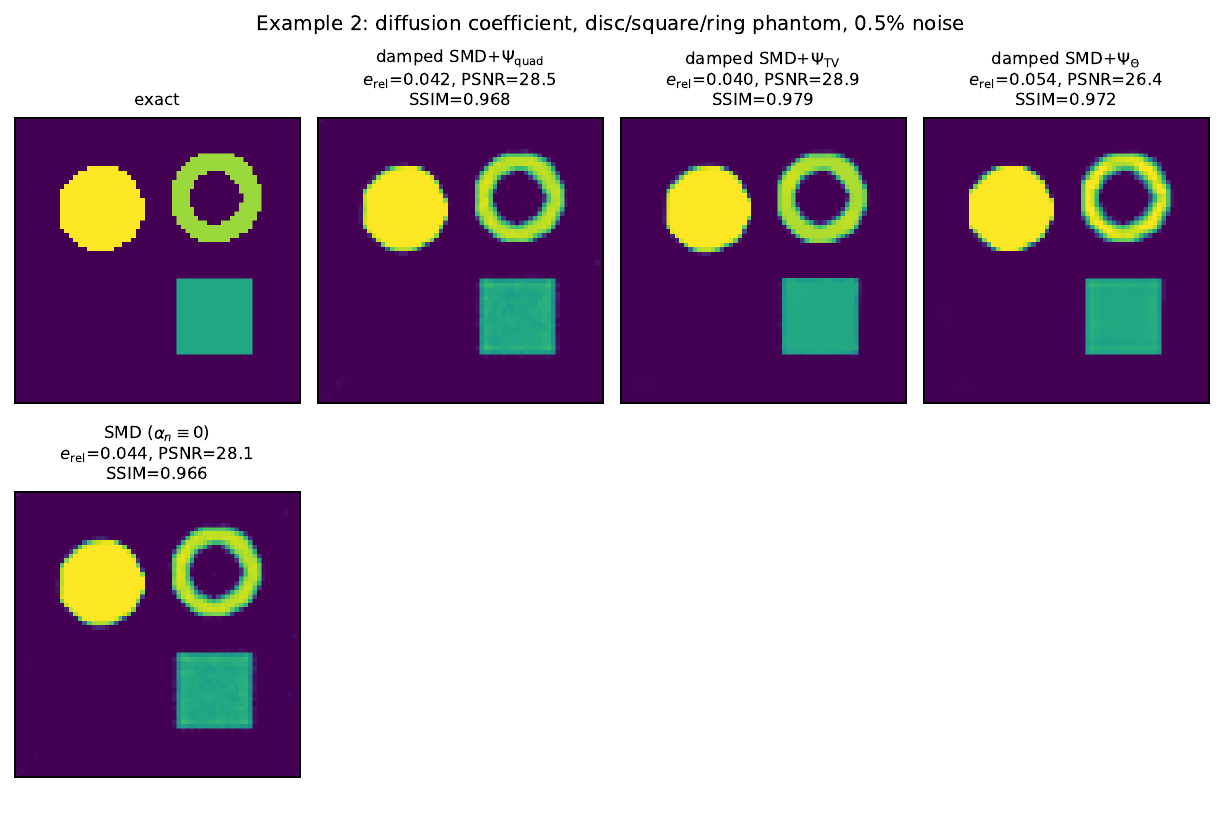}
\caption{Example \ref{ex:diffusion}, $\varrho_{\mathrm{rel}}=0.5\%$: exact diffusion
coefficient and the reconstructions at the stopping index $n_\delta$, for the
Shepp--Logan phantom (top) and the disc/square/ring image (bottom).}
\label{fig:ex3recon}
\end{figure}

\begin{figure}[htbp]\centering
\includegraphics[width=0.8\textwidth]{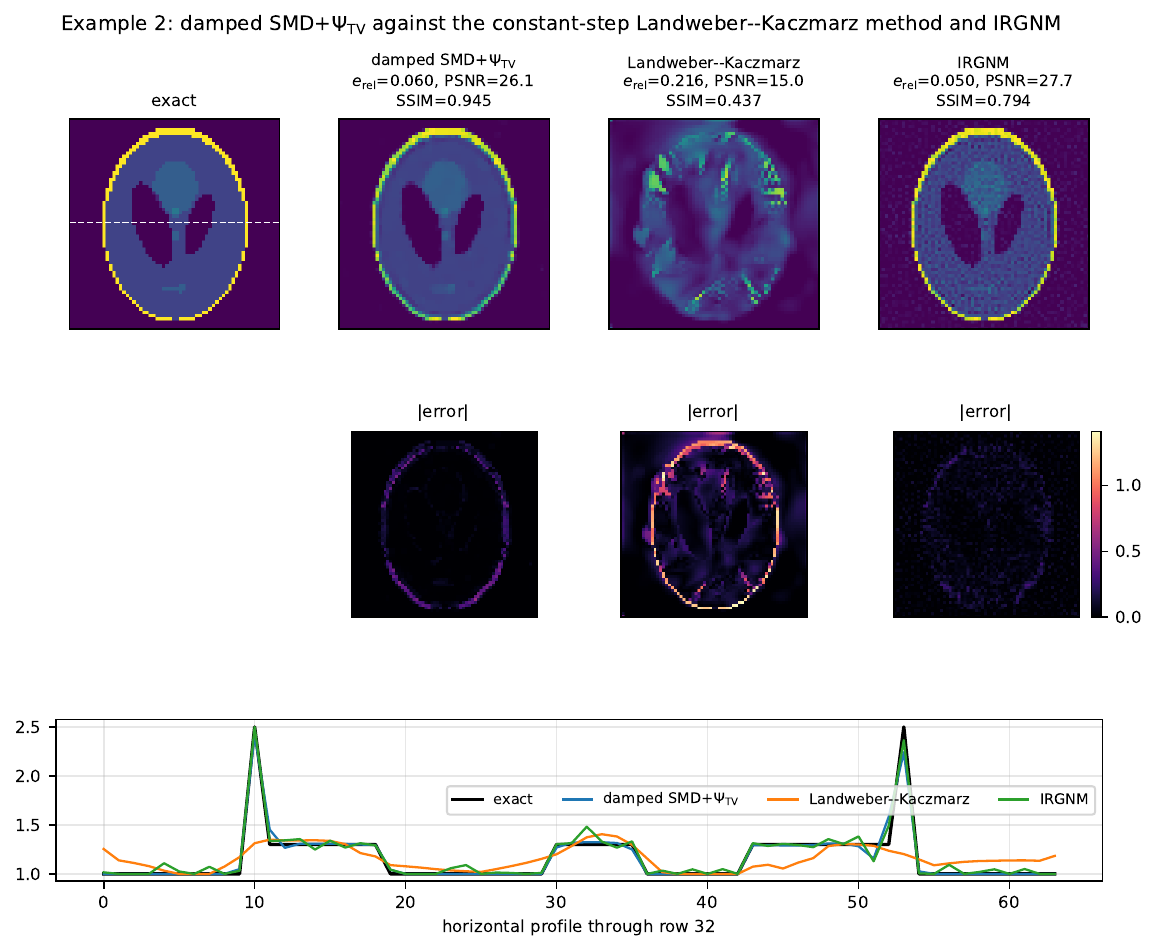}
\caption{Example \ref{ex:diffusion}: d--SMD method~\ref{d--SMD 3.1} with $\Psi_{\mathrm{TV}}$ against
the constant-step Landweber--Kaczmarz method and IRGNM.  Top: reconstructions; middle:
pointwise error on a common scale; bottom: horizontal profile through the middle row.}
\label{fig:ex3bl}
\end{figure}


\subsection{Numerical verification of the assumptions}\label{sec:assumptions-numeric}

Table \ref{tab:assump} verifies, for both problems, the assumptions under which d--SMD method
\ref{d--SMD 3.1} was analysed.

\vspace{0.1cm}
\paragraph{\textbf{Assumption \ref{assumption a}.}}
For the penalty functionals \eqref{eq:psiquad}--\eqref{eq:psitheta}, the theoretical convexity modulus is $c_0=1/\gamma$ with $p=2$. We verify this numerically by evaluating the $p$-convexity inequality on $400$ randomly generated triples $(\overline{u},u,t)$. The computed modulus agrees with the theoretical value up to $12$ significant digits. Since the additional terms $\beta\,\mathrm{TV}$, $\beta\|\Theta\cdot\|_{\ell^1}$, and $\mathbb{I}_{\mathcal{C}}$ are convex, they do not reduce the convexity modulus. Therefore, Assumption~\ref{assumption a} is satisfied with the stated constants.

\vspace{0.1cm}
\paragraph{\textbf{Assumption \ref{assumption b}.}}
\begin{enumerate}
    \item[(i)] For both examples, the domain $\mathbb{D}=\bigcap_i\mathfrak{D}(F_i)$ is an open subset of $\mathcal{U}$, and each forward operator $F_i$ is weakly closed.

    \item[(ii)] Table~\ref{tab:assump} reports the Bregman distance $\mathcal{D}_{\zeta_0}\Psi(u^{\dagger},u_0)$ and the corresponding minimum radius
    \[
    \wp_{\min}=\left(\frac{\mathcal{D}_{\zeta_0}\Psi(u^{\dagger},u_0)}{c_0}\right)^{1/p},
    \]
    for which condition \eqref{sigmacond} is satisfied. In Example~\ref{ex:source}, the inclusion $\mathcal{B}(u_0,3\wp)\subset\mathbb{D}$ holds for all relevant values of $\wp$, so condition \eqref{sigmacond} is automatically satisfied. In Example~\ref{ex:diffusion}, the condition is satisfied for the reported value of $\wp_{\min}$ since $\mathbb{D}\supset\{\sigma>0\}$ and $\sigma^{\dagger}\ge1$.

    \item[(iii)] The operator norm bound $\|L_i(u)\|\le B$ is estimated using the power iteration method. The generalized tangential cone condition \eqref{tangential} is verified numerically using $100$ randomly generated pairs $(\overline{u},u)$ from the ball $\mathcal{B}(u^{\dagger},\wp)$. For the semilinear source problem, we obtain $\eta_{\max}\le2.3\times10^{-6}$, indicating an almost negligible linearization error. For the diffusion problem, we obtain $\eta_{\max}\approx0.57$ and $\eta_{\mathrm{mean}}\approx0.26$, both of which satisfy the required condition $\eta<1$.
\end{enumerate}

\paragraph{\textbf{Admissibility of \eqref{step} and \eqref{condonalpha}.}}
The lower part of Table~\ref{tab:assump} verifies that the conditions on the step size \eqref{step} and the damping parameter \eqref{condonalpha} are satisfied throughout the iteration.



Finally, the finite termination guaranteed by Theorem~\ref{finite iteration} is observed in all numerical experiments.

\newcommand{\dSMD}[1]{d-SMD($#1$)}
\begin{table}[htbp]\centering\small
\caption{Results for Example~\ref{ex:source} with $M=10$, $\ell=1$, $N=47^2$, and $\tau=1.05$. Here, $n_{\mathrm{upd}}$ is the number of iterations with $\omega_n^\delta>0$, i.e., iterations that update the iterate. The quantity $\#\mathcal{S}$ denotes the number of sparse linear solves. Results for the stochastic methods are reported as the mean $\pm$ one standard deviation over independent block-sampling runs. The symbol $^{\dagger}$ indicates that the discrepancy principle was not satisfied within the prescribed iteration limit.}
\label{tab:ex1main}
\resizebox{\textwidth}{!}{%
\begin{tabular}{l cccc cccc cccc}\toprule
& \multicolumn{4}{c}{$\varrho_{\mathrm{rel}}=0.5\%$}
& \multicolumn{4}{c}{$\varrho_{\mathrm{rel}}=1\%$}
& \multicolumn{4}{c}{$\varrho_{\mathrm{rel}}=3\%$}\\
\cmidrule(lr){2-5}\cmidrule(lr){6-9}\cmidrule(lr){10-13}
Method & $n_\delta$ & $n_{\mathrm{upd}}$ & $e_{\mathrm{rel}}$ & $\#\mathcal{S}$
       & $n_\delta$ & $n_{\mathrm{upd}}$ & $e_{\mathrm{rel}}$ & $\#\mathcal{S}$
       & $n_\delta$ & $n_{\mathrm{upd}}$ & $e_{\mathrm{rel}}$ & $\#\mathcal{S}$\\
\midrule
\dSMD{\Psi_{\mathrm{quad}}} & 524 & 512 & 0.194\,$\pm$\,0.000 & 2120 & 211 & 201 & 0.238\,$\pm$\,0.004 & 849 & 86 & 76 & 0.311\,$\pm$\,0.007 & 339\\
\dSMD{\Psi_{\mathrm{TV}}} & 500 & 487 & 0.194\,$\pm$\,0.001 & 2036 & 192 & 182 & 0.234\,$\pm$\,0.006 & 770 & 85 & 75 & 0.291\,$\pm$\,0.014 & 336\\
\dSMD{\Psi_{\Theta}} (trained) & 717 & 706 & 0.199\,$\pm$\,0.001 & 2888 & 276 & 266 & 0.240\,$\pm$\,0.007 & 1111 & 80 & 70 & 0.316\,$\pm$\,0.011 & 322\\
SMD ($\alpha_n\equiv 0$) & 500 & 490 & 0.195\,$\pm$\,0.003 & 2002 & 256 & 246 & 0.239\,$\pm$\,0.005 & 1030 & 86 & 76 & 0.310\,$\pm$\,0.018 & 344\\
Bouligand--Landweber \cite{clason2019bouligand} & 4000$^{\dagger}$ & 4000 & 0.237 & 15997 & 4000$^{\dagger}$ & 2820 & 0.251 & 14817 & 1046 & 497 & 0.334 & 3635\\
IRGNM & 15 & 15 & 0.202 & 6370 & 13 & 13 & 0.253 & 5250 & 11 & 11 & 0.313 & 4130\\
\bottomrule\end{tabular}}\end{table}
\begin{table}[htbp]\centering\small
\caption{Example~\ref{ex:source}: Effect of the block size $\ell=|I_n|$ for $\varrho_{\mathrm{rel}}=1\%$ with the quadratic penalty $\Psi=\Psi_{\mathrm{quad}}$. Increasing the block size leads to only a modest reduction in the number of iterations, while the computational cost per iteration increases approximately linearly.}
\label{tab:ex1ell}
\begin{tabular}{ccccc}\toprule
$\ell$ & $n_\delta$ & $n_{\mathrm{upd}}$ & $e_{\mathrm{rel}}$ & $\#\mathcal{S}$\\\midrule
1 & 216 & 206 & 0.242 & 870\\
2 & 209 & 200 & 0.232 & 1674\\
3 & 179 & 173 & 0.240 & 1996\\
5 & 159 & 156 & 0.243 & 2969\\
10 & 156 & 154 & 0.247 & 5771\\
\bottomrule\end{tabular}\end{table}
\begin{table}[htbp]\centering\small
\caption{Example \ref{ex:source}, ablation of the damping factor.  The admissible
set \eqref{condonalpha} is scaled by $\theta\in[0,1]$
with $\varsigma_4=\varsigma_6=1$; $\theta=0$ is the undamped scheme;
$\varrho_{\mathrm{rel}}=1\%$.}
\label{tab:ex1damp}
\begin{tabular}{cccc}\toprule
$\theta$ & $n_\delta$ & $e_{\mathrm{rel}}$ & $\#\mathcal{S}$\\\midrule
0 & 271 & 0.240\,$\pm$\,0.005 & 1094\\
0.001 & 231 & 0.234\,$\pm$\,0.009 & 936\\
0.003 & 219 & 0.236\,$\pm$\,0.006 & 885\\
0.01 & 210 & 0.240\,$\pm$\,0.002 & 844\\
0.03 & 224 & 0.241\,$\pm$\,0.003 & 901\\
0.1 & 227 & 0.242\,$\pm$\,0.003 & 911\\
\bottomrule\end{tabular}\end{table}
\begin{table}[htbp]\centering\small
\caption{Results for Example~\ref{ex:diffusion} using the Shepp--Logan phantom with $N=64^2$, $M=20$ sources, full-field measurements, block size $\ell=2$, $\varrho_{\mathrm{rel}}=0.5\%$, and $\tau=1.05$. The initial guess is  $\sigma_0\equiv1.5$.  The symbol $^{\dagger}$ indicates that the discrepancy principle \eqref{discrepancy} was not satisfied within the prescribed iteration limit.}
\label{tab:ex2main}
\begin{tabular}{lccccccc}\toprule
method & $n_\delta$ & $n_{\mathrm{upd}}$ & $e_{\mathrm{rel}}$ & PSNR [dB] & SSIM
       & $\#\mathcal{S}$ & CPU [s]\\\midrule
\dSMD{\Psi_{\mathrm{quad}}} & 607 & 592 & 0.065 & 25.47 & 0.892 & 2448 & 3.59\\
\dSMD{\Psi_{\mathrm{TV}}} & 920 & 906 & 0.060 & 26.15 & 0.945 & 3698 & 5.70\\
\dSMD{\Psi_{\Theta}} (trained) & 1846 & 1831 & 0.066 & 25.34 & 0.926 & 7364 & 15.14\\
SMD ($\alpha_n\equiv 0$)  & 589 & 575 & 0.064 & 25.57 & 0.893 & 2348 & 3.27\\
Landweber--Kaczmarz \cite{Kowar2002convergence} & 6000$^{\dagger}$ & 6000 & 0.219 & 14.91 & 0.436 & 11999 & 31.52\\
IRGNM \cite{bakushinskii1992problem} & 25$^{\dagger}$ & 25 & 0.050 & 27.66 & 0.794 & 26200 & 5.35\\
\bottomrule\end{tabular}\end{table}
\begin{table}[htbp]\centering\small
\caption{Example \ref{ex:diffusion} with the second test image (disc, square and
ring); same discretisation and noise level as Table \ref{tab:ex2main}.}
\label{tab:ex2shapes}
\begin{tabular}{lccccccc}\toprule
method & $n_\delta$ & $n_{\mathrm{upd}}$ & $e_{\mathrm{rel}}$ & PSNR [dB] & SSIM
       & $\#\mathcal{S}$ & CPU [s]\\\midrule
\dSMD{\Psi_{\mathrm{quad}}} & 138 & 124 & 0.043\,$\pm$\,0.001 & 28.36 & 0.967 & 562 & 0.71\\
\dSMD{\Psi_{\mathrm{TV}}} & 136 & 124 & 0.042\,$\pm$\,0.001 & 28.57 & 0.978 & 544 & 0.76\\
\dSMD{\Psi_{\Theta}} (trained) & 291 & 280 & 0.054\,$\pm$\,0.000 & 26.32 & 0.972 & 1154 & 2.36\\
SMD ($\alpha_n\equiv 0$) & 132 & 120 & 0.043\,$\pm$\,0.001 & 28.38 & 0.967 & 534 & 0.66\\
Landweber--Kaczmarz \cite{Kowar2002convergence} & 3000$^{\dagger}$ & 2611 & 0.037 & 29.59 & 0.966 & 5610 & 14.14\\
IRGNM \cite{bakushinskii1992problem} & 25$^{\dagger}$ & 25 & 0.033 & 30.62 & 0.826 & 26360 & 5.45\\
\bottomrule\end{tabular}\end{table}
\begin{table}[htbp]\centering\small
\caption{Example \ref{ex:diffusion}, influence of the block size $\ell$
($\Psi=\Psi_{\mathrm{TV}}$, Shepp--Logan phantom).}
\label{tab:ex2ell}
\begin{tabular}{cccccc}\toprule
$\ell$ & $n_\delta$ & $n_{\mathrm{upd}}$ & $e_{\mathrm{rel}}$ & SSIM & $\#\mathcal{S}$\\\midrule
1 & 1202 & 1179 & 0.058\,$\pm$\,0.000 & 0.946 & 2427\\
2 & 914 & 900 & 0.061\,$\pm$\,0.001 & 0.944 & 3676\\
4 & 710 & 700 & 0.065\,$\pm$\,0.001 & 0.938 & 5603\\
8 & 632 & 620 & 0.066\,$\pm$\,0.001 & 0.939 & 9235\\
20 & 406 & 396 & 0.067 & 0.936 & 12739\\
\bottomrule\end{tabular}\end{table}
\begin{table}[htbp]\centering\small
\caption{Numerical verification of Assumptions \ref{assumption a}--\ref{assumption c}
and of the admissibility of \eqref{step} and \eqref{condonalpha}.
$c_0^{\mathrm{emp}}$ is the largest constant for which the $p$-convexity inequality of
Assumption \ref{assumption a} held on $400$ random triples $(\overline{u},u,t)$;
$\eta_{\max}$ is the empirical (GTCC) constant \eqref{tangential} over $100$ random pairs
in $\mathcal{B}(u^{\dagger},\wp)$; $\mathrm{rk}_{10^{-3}}$ is the numerical rank of
$L_1(u_0)$ at relative threshold $10^{-3}$.}
\label{tab:assump}
\begin{tabular}{lccccccc}\toprule
& $c_0$ & $c_0^{\mathrm{emp}}$ & $\mathcal{D}_{\zeta_0}\Psi(u^\dagger,u_0)$ & $\wp_{\min}$
& $B$ & $\eta_{\max}$ & $\mathrm{rk}_{10^{-3}}$\\\midrule
Ex.~\ref{ex:source} & 0.2250 & 0.2250 & 120 & 23.1 & 0.0507 & 2.26e-06 & 2209\\
\midrule
\multicolumn{8}{l}{\emph{along the whole iteration:}}\\
& \multicolumn{2}{c}{bracket of \eqref{step} $>0$}
& \multicolumn{2}{c}{\eqref{reln bw rs and rp} holds}
& \multicolumn{2}{c}{\eqref{condonalpha} holds} & $\max_n\alpha_n$\\
Ex.~\ref{ex:source} & \multicolumn{2}{c}{100\%} & \multicolumn{2}{c}{100\%} & \multicolumn{2}{c}{100\%} & 2.27e-03\\
\bottomrule\end{tabular}\end{table}

\section{Conclusion}\label{sec:Conclusion}
   Stochastic Gradient Descent (SGD) methods have garnered considerable attention in the field of inverse problems due to their simplicity and relatively low computational cost associated with each iteration. Recent advancements have focused on extending the applicability of SGD variants to address inverse problems in Banach spaces. In this context, we introduce a novel method, d--SMD, which significantly broadens the scope of SGD. This approach eliminates the need for the Gâteaux differentiability of the forward operator and the reflexivity of the image space. Moreover, it incorporates a damping term that exhibits an acceleration effect over conventional SGD, as demonstrated in Section \ref{sec:numerics}.
A key contribution of this work lies in establishing regularization properties  by presenting a novel framework for d--SMD \ref{d--SMD 3.1}, under an \textit{a posteriori} stopping rule. Notably, our analysis imposes no additional assumptions on the step size.

In conclusion, several promising directions emerge for future research. First, the discrepancy principle could be further refined by incorporating the approaches developed in~\cite{Huang2025early, verma2026randomized}. In the present framework, the stopping criterion relies only on the local noise level, which may cause the iteration to terminate prematurely and thus represents a potential limitation of the proposed method.
Second, it would be interesting to replace the damping term with a data-driven regularizer, following the ideas in~\cite{aspri2020data,aspri2020datap,bajpai2026convergence,bianchi2025data,bianchi2023uniformly}. Such an extension could provide additional flexibility in incorporating prior information from the data and may further enhance the quality of the reconstructed solutions. Finally, motivated by the approaches in~\cite{Mittal2025,zhong2023landweber}, the concept of CSE could be incorporated into the analysis to establish convergence rates for the proposed d--SMD method~\ref{d--SMD 3.1}, as well as for other SGD-type methods.

 \subsection*{Data Availability} Not applicable for this article.
 \subsection*{Funding statement} First author acknowledges the Anusandhan National Research Foundation (ANRF), India, for supporting his Ph.D. fellowship under grant CRG/2022/005491. Third author gratefully acknowledges ANRF for funding his research through the project CRG/2022/005491.
\subsection*{Author Contributions}
Harshit Bajpai conceptualized the study and drafted the manuscript. Gaurav Mittal and Ankik Kumar Giri provided supervision, critical feedback, and thorough review of the manuscript.
\subsection*{Acknowledgment}
We sincerely thank Dr. Simon Weissmann, University of Mannheim, for his constructive comments, insightful discussions, and valuable suggestions throughout the preparation and refinement of this manuscript. His feedback and suggestions have been instrumental in improving the presentation and strengthening the overall quality of the work.
We also acknowledge the assistance of Claude in code debugging and verification. 

 \vspace{-0.2cm}

\bibliographystyle{abbrv}
\bibliography{references}

\begin{thebibliography}{10}
\small{



 \bibitem{aster2018piandip}{Aster RC, Borchers S and Thurber CH,} \newblock{\em  Parameter estimation and inverse problems}, \newblock{(International Geophysics Series) Elsevier, $2018$.}

 \bibitem{Harshitdatadriven}{Bajpai H, Mittal G and Giri AK,} \newblock{\em   Stochastic Data-Driven Bouligand Landweber Method for Solving Non-smooth Inverse Problems,}\newblock{ J. Inverse Ill-Posed Probl., $33 (2), 153-170$,  $2025$}.

\bibitem{Bajpai2024Hanke}{Bajpai H, Mittal G and Giri AK,} \newblock{\em Hanke-Raus heuristic rule for iteratively regularized stochastic gradient descent,} \newblock{Appl. Numer. Math., $227, 276-298, 2026$.}

\bibitem{bakushinskii1992problem}{Bakushinskii AB,} \newblock{ \em The problem of the convergence of the iteratively regularized Gauss-Newton method,} Computational mathematics and mathematical physics,  1, 32(9), 1353-9, 1992.}
\bibitem{Bao2016}{Bao G, Chen C and Li P,} \newblock{\em Inverse random source scattering problems in several dimensions,} \newblock{SIAM/ASA Journal on Uncertainty Quantification, 4(1):1263-87, 2016.}

\bibitem{Battacharya2007}{Bhattacharya RN, Waymire EC,} \newblock{\em A basic course in probability theory,} \newblock{New York: Springer, 2007.}


\bibitem{boat2012penality}{Boţ RI, Hein T,} \newblock{\em iterative regularization with a general penalty term -- theory and application to L1 and TV regularization }, \newblock{Inverse Probl.,  $1;28(10):104010$, $2012$.}


\bibitem{christof2017optimal}{Christof C,   Clason C,  Meyer C and   Walther S,} \newblock{\em  Optimal control of a non-smooth semilinear elliptic equation,}\newblock{  Math. Control. Relat. Fields, $8$, $247$-$276$,  $2018$}.

\bibitem{clason2012}{Clason C,} \newblock{\em $L^{\infty}$ fitting for inverse problems with uniform noise},\newblock{ Inverse Probl., $28(10):104007$, $2012$.}

\bibitem{clason2019bouligand}{ Clason C and Nhu V,} \newblock{\em Bouligand-Landweber iteration for a non-smooth ill-posed problem},\newblock{ Numer. Math., $142$, $789$-$832$, $2019$.}

\bibitem{clason2019bouligand2}{ Clason C and Nhu V,} \newblock{\em  Bouligand-Levenberg-Marquardt iteration for a non-smooth ill-posed problem},\newblock{ Electron. Trans. Numer. Anal., $51, 274-314,$ $2019$.}  

\bibitem{engl1996regularization}{ Engl HW,  Hanke M and Neubauer A,} \newblock{\em Regularization of inverse problems},\newblock{ Springer Netherlands, $2000$.}

\bibitem{fu2023two}{ Fu Z, Wang W, Han B and Chen Y,} \newblock{\em Two-point Landweber-type method with convex penalty terms for
non-smooth nonlinear inverse problems},\newblock{ IMA Journal of Numer. Anal.,  $1$-$28$, $2021$, doi:10.1093/imanum/drn000.}


 \bibitem{GU2024}{Gu R, Fu Z, Han B and Fu H,} \newblock{\em Stochastic gradient descent method with convex penalty for ill-posed problems in Banach spaces,} \newblock{Inverse Problems, $41(5), 055003, 2025$.}
 
 
 \bibitem{hanafy1991quantitative}{Hanafy A and Zanelli CI,} \newblock{\em Quantitative real-time pulsed Schlieren imaging of ultrasonic waves,} \newblock{IEEE $1991$ Ultrasonics Symposium, $1991$ Dec $8$ (pp. $1223-1227$). IEEE.}

 \bibitem{Hanke}{Hanke M,  Neubauer A and  Scherzer O,} \newblock{\em A convergence analysis of the Landweber iteration for nonlinear ill-posed problems, }\newblock{ Numer. Math., $72$, $21$-$37$, $1995$.}
 

 \bibitem{HeinT2009}{Hein T and Hofmann B,} \newblock{ \em Approximate source conditions for nonlinear ill-posed problems—chances and limitations,} \newblock{Inverse Probl.,  $13;25(3):035003$, $2019$.}

 \bibitem{Huang2025early}{Huang J, Jin Q, Lu X and Zhang L,} \newblock{\em On early stopping of stochastic mirror descent method for ill-posed inverse problems,} \newblock{ Numer. Math., $157$, $539-571$, $2025.$}

 \bibitem{jahn2020discrepancy}{Jahn T and Jin B,} \newblock{\em On the discrepancy principle for stochastic gradient descent,} \newblock{Inverse Probl., $36(9)$, 095009, 2020.}

 \bibitem{jin2018regularizing}{Jin B  and  Lu X,} \newblock{\em On the regularizing property of stochastic gradient descent,} \newblock{ Inverse Probl., $35(1)$, 015004, 2018.}

 \bibitem{jin2020convergence}{Jin B, Zhou Z and Zou J,} \newblock{\em On the convergence of stochastic gradient descent for nonlinear ill-posed problems,} \newblock{SIAM J.  Opti., $30(2),  1421$-$1450$, 2020.}

\bibitem{jin2023}{Jin B and Kereta Ž, \newblock{\em On the convergence of stochastic gradient descent for linear inverse problems in Banach
spaces,} \newblock{ SIAM J. Imaging Sci., $30;16(2):671-705$, $2023$.}

\bibitem{jin2026}{Jin B, Kereta Ž and Xia Y,} Stochastic gradient descent for nonlinear inverse problems in Banach spaces,}  \newblock{ SIAM J. Imaging Sci., $30;19(2):1244-71$, $2026$.} 


\bibitem{qinian2013kaczmarz}{Jin Q and  Wang W,} \newblock{\em  Landweber iteration of Kaczmarz type with general non-smooth convex penalty functionals,} \newblock{Inverse Probl., $29$, no. $8: 085011$, $2013$.}



\bibitem{JinQ2023}{Jin Q, Lu X and Zhang L,} \newblock{\em Stochastic mirror descent method for linear ill-posed problems in Banach spaces,} \newblock{Inverse Probl., $3;39(6):065010$, $2023$.}

\bibitem{JinQ2024-1}{Jin Q and Zhang Y,} \newblock{\em Bouligand–Newton type methods for non-smooth ill-posed problems,} \newblock{Inverse Probl.,40(10), 105011, 2024.}

\bibitem{JinQ2024}{Jin Q,} \newblock{\em Convergence rates of Landweber-type methods for inverse problems in Banach spaces,} \newblock{arXiv preprint arXiv: $2411.09113$, $2024$.}





\bibitem{kaltenbacher2008iterative}{ Kaltenbacher B,  Neubauer A and Scherzer O,} \newblock{\em Iterative regularization methods for nonlinear Ill-posed problems},\newblock{ De Gruyter, $2008$.}


\bibitem{kaltenbacher2013crates}{Kaltenbacher B,} \newblock{\em Convergence rates for the iteratively regularized Landweber iteration in Banach space}, \newblock{ InSystem Modeling and Optimization: 25th IFIP TC 7 Conference, CSMO $2011$, Berlin, Germany, September $12-16, 2011,$ Revised Selected Papers $25 2013$ (pp. $38-48$). Springer Berlin Heidelberg. , $2012$.}

\bibitem{kikuchi1984finite}{Kikuchi F,   Nakazato K and   Ushijima T,} \newblock{\em Finite element approximation of a nonlinear eigenvalue problem related to MHD equilibria},\newblock{ Japan J. Appl. Math., $1$, $369$-$403$,  $1984$.}

\bibitem{Kowar2002convergence} {Kowar R and Scherzer O,} \newblock{Convergence analysis of a Landweber-Kaczmarz method for solving nonlinear ill-posed problems,} \newblock{Ill-posed and inverse problems, $23, 253-70, 2002.$}
\bibitem{LI2023}{Li P, Liang Y and Wang Y,} \newblock{\em A data-assisted two-stage method for the inverse random source problem,} \newblock{SIAM J. Imaging Sci., $31, 16(4):1929-52, 2023.$}

\bibitem{Long2024} {Long H, Zhang Y and Gao G,} \newblock{\em An accelerated inexact Newton regularization scheme with a learned feature-selection rule for non-linear inverse problems,} \newblock{Inverse Probl.,  $11;40(8):085011$, $2024$.}


\bibitem{lu2022stochastic}{Lu S and Math\'e P,} \newblock{\em Stochastic gradient descent for linear inverse problems in Hilbert spaces,} \newblock{ Math.   Comp.,  $91(336),1763$-$1788$, 2022.}

\bibitem{Mittal2023}{Mittal G and Giri AK,} \newblock{\em A modified steepest descent method for solving non-smooth inverse problems,} \newblock{J. Comput. Appl. Math., 424, 114997, 2023.}

\bibitem{Mittal2025}{Mittal G, Bajpai H and Giri AK,} \newblock{\em Convergence analysis of iteratively regularized Landweber iteration with uniformly convex constraints in Banach spaces,} \newblock{J. Complex.,  $1;86:101897$, $2025$.}

\bibitem{natterer2001mathematics}{Natterer F}, \newblock{\em The mathematics of computerized tomography,} \newblock{Society for Industrial and Applied Mathematics, $2001$.}

\bibitem{Neuber1995}{Neubauer A, Scherzer O,} \newblock{\em A convergence rate result for a steepest descent method and a minimal error method for the solution of nonlinear ill-posed problems,} \newblock{Z. Anal.  Anwendungen,  $30;14(2):369-77$, $1995$.}

\bibitem{Neuber2018}{Neubauer A,} \newblock{\em A new gradient method for ill-posed problems,} \newblock{Numer. Funct. Anal. Optim., $26;39(6):737-62$, $2018$.}

\bibitem{rappaz1984approximation}{  Rappaz J,} \newblock{\em Approximation of a non differentiable nonlinear problem related to MHD equilibria},\newblock{ Numer. Math., $45$, $117$-$133$,  $1984$.}

\bibitem{rreal2020revisit}{Real R and Jin Q,} \newblock{\em A revisit on Landweber iteration,} \newblock{Inverse Probl., $19;36(7):075011$, $2020$.}

\bibitem{real2024hanke}{Real R}, \newblock{\em Hanke–Raus rule for Landweber iteration in Banach spaces,} \newblock{Numer. Math., $156(1),345$-$373$, $2024$.}

\bibitem{robbins1951stochastic}{Robbins H and Monro S}, \newblock{\em A stochastic approximation method,} \newblock{Ann. Math. Stat., 1, 400-407, 1951.}

\bibitem{Scherzer1996}{Scherzer O,} \newblock{\em A convergence analysis of a method of steepest descent and a two–step algorithm for nonlinear ill–posed problems: A two–step algorithm for nonlinear ill–posed problems,} \newblock{Numer. Funct. Anal. Optim.,  $1, 17(1-2), 197-214$, $1996$.}

\bibitem{Scherzer1998}{Scherzer O,} \newblock{\em A modified Landweber iteration for solving parameter estimation problems},\newblock{ Appl. Math. Optim.,  $38$, $45$-$68$,  $1998$.}

\bibitem{Schoper2006}{Schöpfer F, Louis AK and Schuster T,} \newblock{\em Nonlinear iterative methods for linear ill-posed problems in Banach spaces,} \newblock{ Inverse probl., $30, 22(1), 311$, $2006$.}

\bibitem{kaltenbacher2012Banach}{Schuster T, Kaltenbacher B, Hofmann B and Kazimierski KS,} \newblock{\em Regularization methods in Banach spaces (Vol. 10)}, \newblock{  Walter de Gruyter, $2012$.}


\bibitem{tong2021}{Tong S, Wang W and Han B,} \newblock{\em Accelerated homotopy perturbation iteration method for a non-smooth nonlinear ill-posed problem,} \newblock{Appl. Numer. Math., 169, 122-145, 2021.}

\bibitem{zalinesch202convex}{Zalinescu C,} \newblock{\em Convex analysis in general vector spaces, } \newblock{World scientific, $2002$.}


\bibitem{zhong2023landweber}{Zhong M, Qiu L and Wang W,} \newblock{\em Landweber-type method with uniformly convex constraints under conditional stability assumptions,} \newblock{Appl. Math. Lett., $1, 144, 108723$, $2023$.}
\small
 \end{thebibliography}

\end{document}